\documentclass{amsart}

\usepackage{amssymb}
\usepackage{amscd}
\usepackage{graphicx}
\usepackage{pstricks}
\usepackage[cmtip,all]{xy}  %commutative diagrams
\usepackage{enumitem}

\newtheorem{theorem}{Theorem}[section]
\newtheorem{lemma}[theorem]{Lemma}
\newtheorem{prop}[theorem]{Proposition}
\newtheorem{cor}[theorem]{Corollary}
\newtheorem*{theoremA}{Theorem A}
\newtheorem*{theoremB}{Theorem B}

\theoremstyle{definition}
\newtheorem{definition}[theorem]{Definition}

\theoremstyle{remark}

\newtheorem{obs}[theorem]{Observation}

\numberwithin{equation}{section}

\begin{document}
\title{Smale Spaces via Inverse Limits}
\author{Susana Wieler}
\address{Dept. of Mathematics and Statistics, University of Victoria,  3800 Finnerty Road, Victoria, B.C. Canada V8P 5C2}
\email{susanaw@uvic.ca}


\date{June 4, 2012}

\begin{abstract}
A Smale space is a chaotic dynamical system with canonical coordinates of contracting and expanding directions.  The basic sets for Smale's Axiom $A$ systems are a key class of examples.  
We consider the special case of irreducible Smale spaces with zero dimensional contracting directions, and characterize these as stationary inverse limits satisfying certain conditions.  
\end{abstract}

\maketitle

\section{Introduction}

D. Ruelle \cite{Ruelle} defined Smale spaces in an effort to axiomatize the topological dynamics of the basic sets of an Axiom $A$ system.    The idea of moving from an Axiom $A$ system to a Smale space is motivated by the fact that the basic sets themselves are merely topological spaces and not submanifolds.  The main feature of a Smale space is its canonical coordinates of contracting and expanding directions, which are defined using a metric rather than differential topology.  

It is well-known that all totally disconnected Smale spaces are shifts of finite type.  And shifts of finite type are inverse limits of one-sided shifts of finite type, which were characterized by W. Parry \cite{Parry} as positively expansive open mappings of compact, totally disconnected metrizable spaces.  The natural next step is to consider Smale spaces which are totally disconnected in only one coordinate and to work towards a characterization of these as inverse limits of spaces satisfying certain conditions.  

R.F. Williams \cite{Williams} looked at expanding attractors.  He proved that these are basic sets which are totally disconnected in the contracting direction and Euclidean in the expanding direction.  He provided a construction using inverse limits of branched manifolds and also proved that (under appropriate hypotheses) all such basic sets arose from this construction.  We will be working in the metric setting of Smale spaces, but the goal is to extend Williams' results by removing all hypotheses on the unstable sets.  Williams relied very heavily on the smooth structures of branched manifolds in his conditions and proofs, and to adapt to the metric setting of Smale spaces, we really needed a whole new set of ideas.  We do not simply ignore the differentiable structure of Williams' systems and adapt his arguments accordingly; such a development, in dimension 1, is given in \cite{Yi}.

This paper is a summary of the my Ph.D. thesis \cite{thesis}.  It is a pleasure to thank my advisor Ian Putnam for many useful discussions.  In particular, Example 3 in Section \ref{results} was suggested by him.

\section{Background and Statement of Results}\label{results}

To provide some intuition for our results, we begin this section with a very brief review of Williams' conditions for his inverse limit spaces. This is followed by a proper definition of a Smale space and a statement of our conditions and results.  We finish with a number of examples to illustrate our conditions.  

\subsection{Williams' Inverse Limits}\label{Williams}
R. F. Williams \cite{Williams} defined an $n$-solenoid as an inverse limit
\[ \hat{K}=\underleftarrow{\mathrm{lim}}\; K \stackrel{g}{\longleftarrow} K \stackrel{g}{\longleftarrow} \cdots, \]
where $K$ is a compact Riemannian branched $C^r$ $n$-manifold and $g:K\rightarrow K$ is a $C^r$ immersion satisfying the following axioms:
\begin{enumerate}
\item $g$ is non-wandering,
\item $g$ is an expansion: there exist constants $A>0$ and $\mu>1$ such that for all $n\in \mathbb{N}$ and $k\in T(K)$, we have
$|Dg^n(k)|\geq A\mu^n|k|$, where $T(K)$ is the tangent space of $K$ and $Dg$ is the derivative of $g$, and  
\item $g$ is flattening: for each $x\in K$ there is a neighborhood $N$ of $x$ and $j\in \mathbb{Z}$ such that $g^j(N)$ is contained in a subset diffeomorphic to an open ball in $\mathbb{R}^n$.
\end{enumerate}
He proves that an $n$-solenoid locally has the structure of a (Cantor set)$\times$($n$-disk). Moreover, he proves that under certain conditions (which were later removed by H.G. Bothe in \cite{Bothe}) his expanding attractors are conjugate to $n$-solenoids.  

Intuitively, Williams' expansive and flattening conditions are contradictory.  This apparent dilemma is solved because it is the derivative, $Dg$, that is injective and expanding, while $g$ is flattening.  Since we are approaching this problem outside of the smooth structure of manifolds, we need entirely different conditions.

\subsection{Statement of Results}
\begin{definition}\label{Smale}
Let $(X,d)$ be a compact metric space, and let $f:X \rightarrow X$ be a homeomorphism.  For $x\in X$ and $\epsilon>0$, we denote 
\[X^s(x, \epsilon)= \lbrace y\in X \; | \; d(f^n(x), f^n(y))\leq \epsilon, \; n \geq 0 \rbrace\]
and
\[X^u(x, \epsilon)= \lbrace y\in X \; | \; d(f^{-n}(x), f^{-n}(y))\leq \epsilon, \; n \geq 0 \rbrace;\]
these are called the local stable and unstable sets of $x$, respectively.  The triple $(X,d,f)$ is a \textit{Smale space}  if there exist constants $\epsilon_X>0$, $\epsilon_X'>0$ and $0<\lambda <1$ such that
\begin{enumerate}
\item for all $x\in X$ and $n\geq 0$, $d(f^n(y), f^n(z))\leq \lambda^n d(y,z)$ if $y,z \in X^s(x, \epsilon_X')$,
\item for all $x\in X$ and $n\geq 0$, $d(f^{-n}(y), f^{-n}(z))\leq \lambda^n d(y,z)$ if $y,z \in X^u(x, \epsilon_X')$, and
\item if $d(x,y)\leq \epsilon_X$ then $X^s(x, \epsilon_X')\cap X^u(y, \epsilon_X')$ consists of a single point, denoted by $[x,y]$.  
\end{enumerate}
\end{definition}

The bracket map $[\cdot, \cdot ]: \lbrace (x,y)\in X\times X \; | \; d(x,y)\leq \epsilon_X \rbrace \rightarrow X$ is continuous.  Moreover, for small enough $\epsilon>0$, $[\cdot, \cdot ]$ restricted to $X^u(x, \epsilon)\times X^s(x,\epsilon)$ is a homeomorphism onto a neighborhood of $x$.  

\medskip

We want to show that certain Smale spaces are inverse limits.  To this end, we will construct stationary inverse limits of spaces satisfying the following two conditions, where the notation $B(x,r)$ denotes a \textit{closed} ball.  

Let $(Y,d)$ be a compact metric space, and let $g:Y\rightarrow Y$ be continuous and surjective.  
We will say that $(Y,d,g)$ satisfies Axioms 1 and 2 if there exist constants $\beta>0$, $K\geq 1$, and $0<\gamma < 1$ such that
\begin{description}
\item[Axiom 1]{\textit{if $d(x,y)\leq\beta$ then $$d(g^K(x),g^K(y))\leq\gamma^K d(g^{2K}(x), g^{2K}(y)),$$ and}}
\item[Axiom 2]{\textit{for all $x\in Y$ and $0<\epsilon\leq\beta$, $$g^K(B(g^K(x), \epsilon))\subseteq g^{2K}(B(x, \gamma\epsilon)).$$}}
\end{description}
Intuitively, Axioms 1 and 2 could be viewed as weakend versions of the conditions that $g$ be locally expanding and open, respectively.  
Locally expanding would be Axiom 1 if the inequality were replaced with $d(x, y)\leq\gamma^K d(g^{K}(x), g^{K}(y))$.  And Axiom 2, with the
containment replaced by $B(g^K(x), \epsilon)\subseteq g^{K}(B(x, \gamma\epsilon))$, would imply that $g^K$ (and hence $g$) is open.

We denote the stationary inverse limit associated with $(Y,g)$ by 
\begin{align*}
\hat{Y}& =\underleftarrow{\mathrm{lim}}\; Y \stackrel{g}{\longleftarrow} Y \stackrel{g}{\longleftarrow} Y \stackrel{g}{\longleftarrow} \cdots\\
&=\lbrace (y_0, y_1, y_2, \cdots ) \; | \; y_n\in Y, \; y_n=g(y_{n+1}) \; \forall \; n\geq 0\rbrace,
\end{align*}
and the natural mapping on $\hat{Y}$ induced by $g$ is denoted
\[\hat{g}:(y_0,y_1, y_2, \cdots ) \mapsto (g(y_0),g(y_1), g(y_2), \cdots )=(g(y_0),y_0, y_1, \cdots ).\]
We will use $d$ to define a metric $\hat{d}$ on $\hat{Y}$.

The following two theorems are the main results of this paper, and their proofs make up Sections \ref{thmApf} and \ref{thmBpf}, respectively.

\begin{theoremA}
If $(Y,d,g)$ satisfies Axioms 1 and 2 then $(\hat{Y}, \hat{d}, \hat{g})$ is a Smale space with totally disconnected local stable sets.  Moreover, $(\hat{Y}, \hat{d}, \hat{g})$ is an irreducible Smale space if and only if $(Y,d,g)$ is non-wandering and has a dense forward orbit.  
\end{theoremA}

\begin{theoremB}
Let $(X,d,f)$ be an irreducible Smale space with totally disconnected local stable sets.  Then $(X,d,f)$ is topologically conjugate to an inverse limit space $(\hat{Y}, \hat{\delta}, \hat{\alpha})$ such that $(Y, \delta, \alpha)$ satisfies Axioms 1 and 2.  
\end{theoremB}

The space $Y$ in Theorem B is a quotient of $X$, and the proof of the theorem is constructive.  If apply the construction to an irreducible shift of finite type, the quotient is simply the one-sided shift and the inverse limit recovers the original two-sided shift.  

It is clear that Axioms 1 and 2 are central and subtle.  The following three examples are intended to give some context for these conditions.  

\subsection*{Example 1:  The role of the constant $K\geq 1$}

The following well-known example satisfies R.F. Williams' conditions for a 1-solenoid \cite{Williams67}, as well as our Axioms 1 and 2.  

Let $Y$ be a wedge of two circles, $a$ and $b$, joined at a single point $v$.
Let both circles have circumference 1.  

Divide $a$ into thirds and $b$ into halves.  
Let $g:Y\rightarrow Y$ map each of the first two intervals of $a$ onto $a$ and the third interval onto $b$; the first interval of $b$ onto $a$ and the second onto $b$ (see Figure \ref{g}).

\begin{figure}[h!]
\begin{center}
\begin{pspicture}(-2,-2.3)(2,2)
\pscircle(0,0){2}
\pscircle(0,-.5){1.5}
\psdot(0,-2)
\psdot(0,1)
\psdot(-1.732,1)
\psdot(1.732,1)
\rput(0,2.3){$a$}
\rput(-2.2,-.5){$a$}
\rput(2.2,-.5){$b$}
\rput(-1.2,-.5){$a$}
\rput(1.2,-.5){$b$}
\rput(0,-2.2){$v$}
\end{pspicture}
\end{center}
\caption{$a\mapsto aab$, $b\mapsto ab$}
\label{g}
\end{figure}
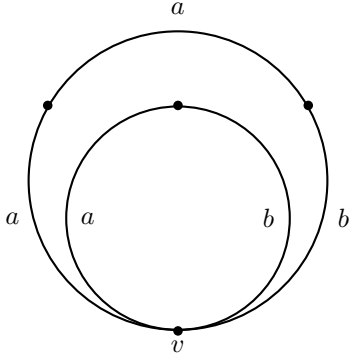

This example does not satisfy the locally expanding version of Axiom 1:  for any $\beta >0$, we can find $x\in a\setminus \lbrace v \rbrace$ and $y\in b\setminus \lbrace v \rbrace$ such that $d(x,y)\leq \beta$ and $g(a)=g(b)$.  So we cannot have $d(x, y)\leq\gamma^K d(g^{K}(x), g^{K}(y))$ for any $K\geq 1$ and $0<\gamma <1$.  

What is happening here is that on a local level, the first iteration of $g$ flattens, and the second iteration expands.  We remark that Williams' conditions allow for both flattening and expansion because it is the derivative of the map that is expanding.

\subsection*{Example 2:  A surjection failing Axiom 2}

 Let $\Sigma^+_{\lbrace 0,1\rbrace}$ and $\Sigma^+_{\lbrace 0,2\rbrace}$ be the full one-sided shifts on the symbol sets $\lbrace 0,1 \rbrace$ and $\lbrace 0,2 \rbrace$, respectively.  We'll use the following common metric on these one-sided shifts:  $d(\mathbf{x}, \mathbf{y})=2^{-\min\lbrace n \; | \; x_n \neq y_n \rbrace}$.  

Let $Y= \Sigma^+_{\lbrace 0,1\rbrace} \bigcup \Sigma^+_{\lbrace 0,2\rbrace}$, and $g$ be the usual left shift map.  Then $g$ is clearly a continuous and surjective map on $Y$.  

Let us show that Axiom 2 fails for $(Y,d,g)$.  Choose $K\geq 1$, $N\geq 2K$ and $0< \gamma < 1$.  Consider the points $\mathbf{x}, \mathbf{y} \in Y$  given by  
\[ x_n = \left\lbrace 
\begin{array}{ll}
1 & \mathrm{if} \; n=N+K \\
0 & \mathrm{if} \; n\neq N+K
\end{array}
 \right. \]
and
\[ y_n = \left\lbrace 
\begin{array}{ll}
2 & \mathrm{if} \; n=N \\
0 & \mathrm{if} \; n\neq N
\end{array}
 \right. . \]
Then $d(g^K(\mathbf{x}), \mathbf{y})=2^{-N}$.  However, $g^K(\mathbf{y})\notin g^{2K}(B(\mathbf{x}, \gamma 2^{-N}))$ since for any point $\mathbf{z}\in B(\mathbf{x}, \gamma 2^{-N})$ we have $g^{2K}(\mathbf{z})_{N-2K}=z_N=x_N=0$.  

It is easy to see that the inverse limit $(\hat{Y}, \hat{g})$ is conjugate to $(\Sigma_{\lbrace 0,1\rbrace} \bigcup \Sigma_{\lbrace 0,2\rbrace}, S)$, where $\Sigma_{\lbrace 0,1\rbrace}$ and $\Sigma_{\lbrace 0,2\rbrace}$ are the full two-sided shifts on their respective symbol sets, and $S$ is the left shift map.  However this system is not a Smale space: we can find distinct points $\mathbf{x}\in \Sigma_{\lbrace 0,1\rbrace}$  and $\mathbf{y}\in \Sigma_{\lbrace 0,2\rbrace}$ that are arbitrarily close, yet the intersection of the local stable set of $\mathbf{x}$ (contained entirely in $\Sigma_{\lbrace 0,1\rbrace}$) and the local unstable set of $\mathbf{y}$ (contained entirely in $\Sigma_{\lbrace 0,2\rbrace}$) is empty.  

\subsection*{Example 3: A space satisfying Axioms 1 and 2}

Consider the following example.  Take six copies, $Y_1, Y_2, \cdots , Y_6$, of the Sierpinski gasket with distinguished vertices, as in Figure \ref{six}.  (For a description of the Sierpinski gasket, see \cite{Falconer}).  We note that the picture omits the intricate structure in the interior of the triangles.  Moreover, the short extensions added to the vertices ought not to be considered as points in the space, but merely as labels which code the same information as $A$, $B$, and $C$.  

Let $\sim$ be the equivalence relation on $Y_1 \cup Y_2 \cup \cdots \cup Y_6$ identifying the six vertices labeled $A$, the six vertices labeled $B$, and the six vertices labeled $C$.  That is, the only equivalence classes containing more than one point are $A$, $B$, and $C$.  We define
\[Y= \left(  Y_1 \cup Y_2 \cup \cdots \cup Y_6 \right) /_\sim.\]

\begin{figure}[h!]
\psset{xunit=.7cm,yunit=.7cm}
\begin{pspicture}(0,-0.2)(15,10)
\includegraphics[scale=.7]{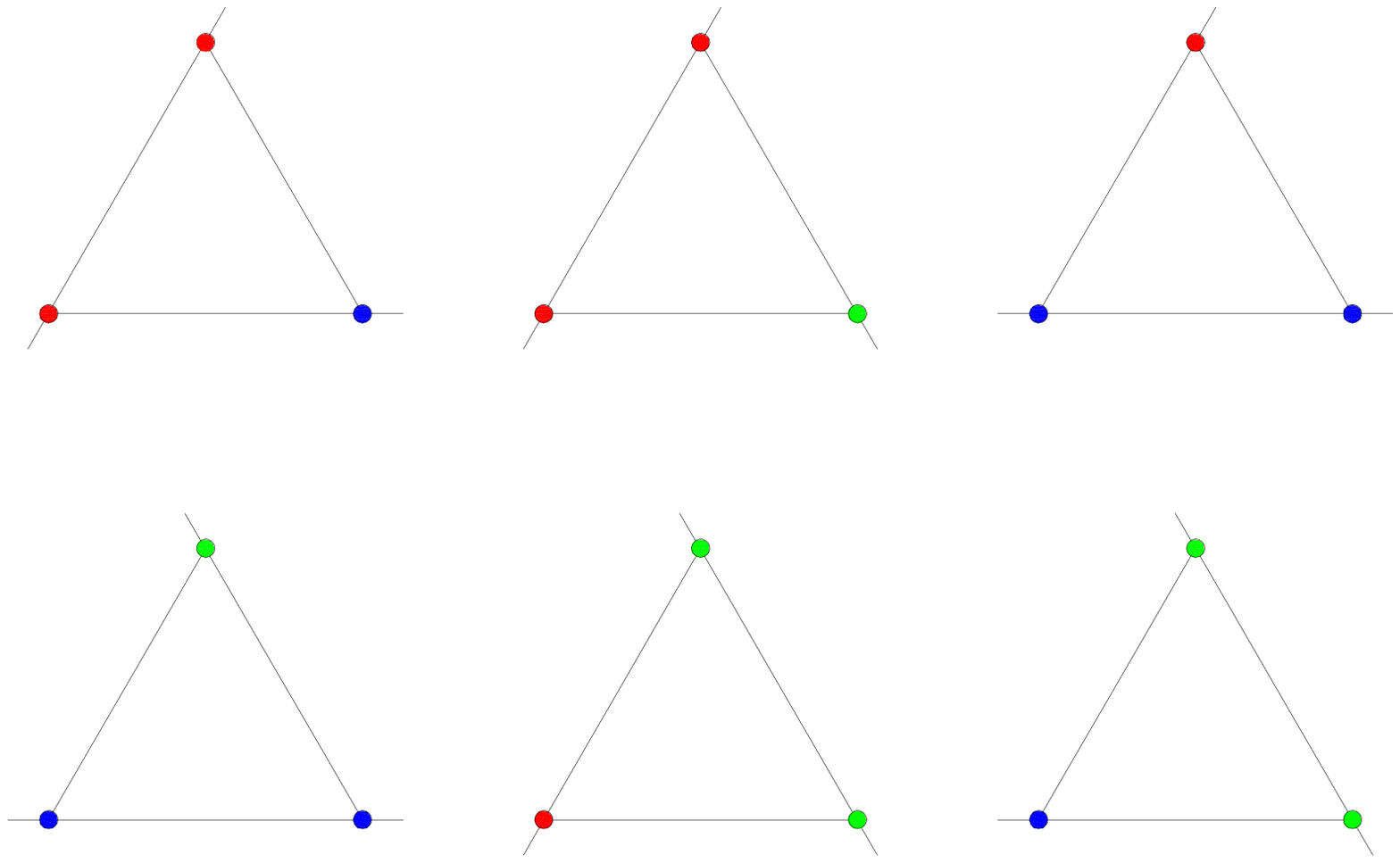}
\rput(-15.6,9.5){$A$}
\rput(-10,9.5){$A$}
\rput(-4.4,9.5){$A$}
\rput(-17.4,6.6){$A$}
\rput(-13.3,6){$B$}
\rput(-11.7,6.6){$A$}
\rput(-7.3,6.6){$C$}
\rput(-5.8,6){$B$}
\rput(-1.9,6){$B$}
\rput(-15.6,3.7){$C$}
\rput(-17.2,0.2){$B$}
\rput(-13.3,0.2){$B$}
\rput(-10,3.7){$C$}
\rput(-11.7,0.5){$A$}
\rput(-7.3,.5){$C$}
\rput(-4.4,3.7){$C$}
\rput(-5.8,.2){$B$}
\rput(-1.6,.5){$C$}
\rput(-15,5.7){$Y_1$}
\rput(-9.5,5.7){$Y_2$}
\rput(-3.7,5.7){$Y_3$}
\rput(-15,-0.1){$Y_4$}
\rput(-9.5,-0.1){$Y_5$}
\rput(-3.7,-0.1){$Y_6$}
\end{pspicture}
\caption{Three distinguished vertices}
\label{six}
\end{figure}

We use the standard ``shortest path" metric on $Y$.  See Figure \ref{A} for an example of a neighborhood of $A$.

\begin{figure}[h!]
\begin{center}
\includegraphics[scale=.5]{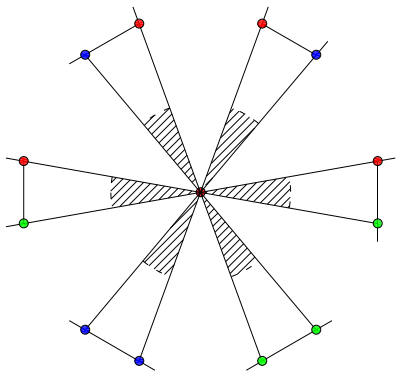}
\end{center}
\caption{A neighborhood of $A$}
\label{A}
\end{figure}

It is clear that $(Y,d)$ is compact.

We define a mapping $g:Y\rightarrow Y$ as follows: $g$ fixes $A$, $B$, and $C$.  In each triangle $Y_i$, the midpoint of the left edge is mapped to $A$, the midpoint of the bottom edge is mapped to $B$, and the midpoint of the right edge is mapped to $C$ (hence the labels as small line segments).  On each of the remaining three small gaskets which make up $Y_i$, $g$ scales by a factor of 2 and maps them onto the unique gasket with the vertices as specified by the images of the three corners.  
As an example, the image of $Y_1$ is shown in Figure \ref{gY1}.  

\begin{figure}[h!]
\psset{xunit=.8cm,yunit=.8cm}
\begin{pspicture}(-.5,0)(5,7)
\includegraphics[scale=.6]{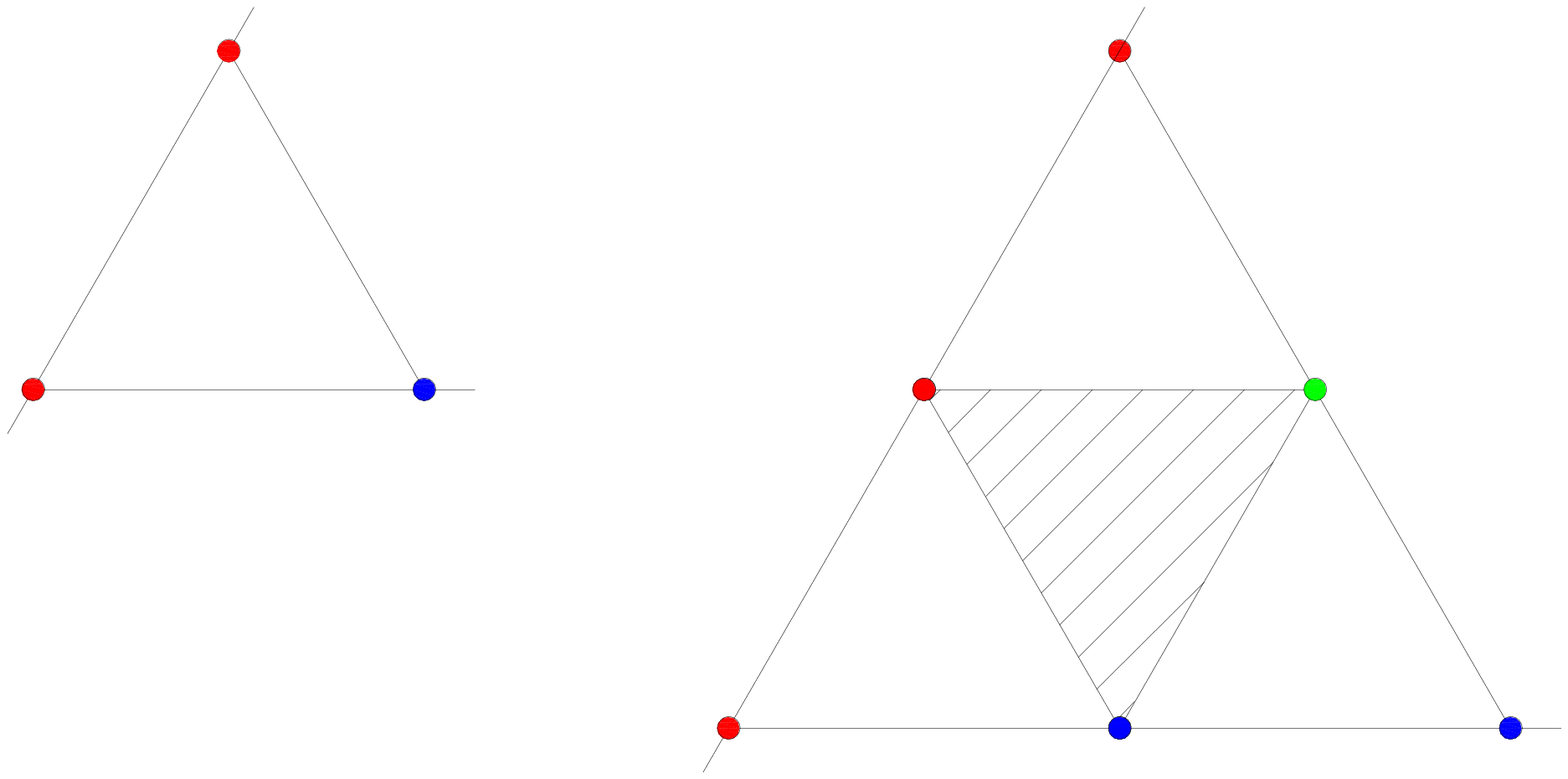}
\rput(-12.3,5){$Y_1$}
\rput(-14.4, 4){$A$}
\rput(-12.6, 7){$A$}
\rput(-10.2, 3.3){$B$}
\rput(-4,5){$Y_2$}
\rput(-5.9,2){$Y_1$}
\rput(-2.2,2){$Y_4$}
\rput(-4.5, 7){$A$}
\rput(-7.9, 1){$A$}
\rput(-0.3, 0.2){$B$}
\rput(-4, 0.2){$B$}
\rput(-6.4, 3.6){$A$}
\rput(-1.9, 3.6){$C$}
\psline[arrowscale=2]{->}(-8.5,4.5)(-7.5,4.5)
\rput(-8,4.8){$g$}
\end{pspicture}
\caption{$g(Y_1)$}
\label{gY1}
\end{figure}

The relation $\sim$ ensures that $g$ is well-defined on $Y$.  This mapping is clearly continuous on $Y$. Observe, however, that the map $g$ is not locally injective, and hence not locally expanding either. Moreover, for every $k\geq 1$, the map $g^k$ is not an open map.  

It is not hard to see that $(Y,d,g)$ satisfies Axiom 1, since $g$ is essentially a scale-and-subdivide mapping.  The subtlety of this example lies in what happens at the vertices.  The image of a small enough neighborhood of any vertex, as shown in Figure \ref{A}, intersects only two of the gaskets, and it is this notion of flattening that gives us Axiom 2.

\section{Proof of Theorem A}\label{thmApf}

Suppose that $(Y,d,g)$, together with the constants $\beta>0$, $K\geq 1$, and $0<\gamma < 1$, satisfies Axioms 1 and 2.

We define a metric $\hat{d}$ on $\hat{Y}$ by 
\[\hat{d}(\mathbf{x}, \mathbf{y})= \sum_{k=0}^{K-1}\gamma^{-k}d'(\hat{g}^{-k}(\mathbf{x}),\hat{g}^{-k}(\mathbf{y})), \]
where $d'(\mathbf{x}, \mathbf{y})=\mathrm{sup}_{n\geq 0}\lbrace \gamma^n d(x_n,y_n)\rbrace$.  It is not hard to show that $\hat{d}$ gives the product topology on $\hat{Y}$, so that $(\hat{Y}, \hat{d})$ is compact.  Moreover, it is clear that the map $\hat{g}$ is a homeomorphism on $(\hat{Y}, \hat{d})$.  

Our first task will be to obtain more useful descriptions of the sets $\hat{Y}^s(\mathbf{x}, \epsilon)$ and $\hat{Y}^u(\mathbf{x}, \epsilon)$.

We make the following easy observation about $d'$.  

\begin{obs}\label{d'}
If $x_0=y_0$ then $d'(\hat{g}(\mathbf{x}), \hat{g}(\mathbf{y}))=\gamma d'(\mathbf{x}, \mathbf{y})$.  
\end{obs}

Choose 
$$0<\epsilon_{\hat{Y}}' \leq \frac{1}{2}\beta$$ 
such that $\hat{d}(\mathbf{x}, \mathbf{y})\leq \epsilon_{\hat{Y}}'$ implies $\hat{d}(\hat{g}^{-n}(\mathbf{x}), \hat{g}^{-n}(\mathbf{y}))\leq \beta$ for $n=0, \cdots, 2K-1$. 

\begin{lemma}\label{x0=y0}  
For any $0<\epsilon \leq \epsilon_{\hat{Y}}'$, $\mathbf{y}\in \hat{Y}^s(\mathbf{z}, \epsilon)$ if and only if $y_m=z_m$ for $m=0, \cdots, K-1$ and $\hat{d}(\mathbf{y}, \mathbf{z})\leq\epsilon$.
\end{lemma}

\begin{proof}
First, suppose that $\mathbf{y}\in \hat{Y}^s(\mathbf{z}, \epsilon)$.  By our choice of $\epsilon_{\hat{Y}}'$, we have 
$$\hat{g}^{-(2K-1)}(\mathbf{y})\in \hat{Y}^s(\hat{g}^{-(2K-1)}\mathbf{z}, \beta).$$  
So for each $m=0, \cdots, K-1$ and any $n\geq 0$, we have 
\begin{align*}
d(g^n(y_{K+m}), g^n(z_{K+m}) & = d(g^{n}(\hat{g}^{-(K+m)}(\mathbf{y})_0), g^{n}(\hat{g}^{-(K+m)}(\mathbf{z})_0)) \\
&=d(\hat{g}^{n-(K+m)}(\mathbf{y})_0, \hat{g}^{n-(K+m)}(\mathbf{z})_0)\\
&\leq d'(\hat{g}^{n-(K+m)}(\mathbf{y}), \hat{g}^{n-(K+m)}(\mathbf{z}))\\
&\leq \hat{d}(\hat{g}^{n-(K+m)}(\mathbf{y}), \hat{g}^{n-(K+m)}(\mathbf{z}))\\
& \leq \beta. 
\end{align*}
Applying Axiom 1, we get 
$$d(g^{K+n}(y_{K+m}), g^{K+n}(z_{K+m}))\leq \gamma^K d(g^{2K+n}(y_{K+m}), g^{2K+n}(z_{K+m}))$$
for all $n\geq 0$.  
That is, 
\begin{align*}
d(y_m,z_m) & =d(g^{K}(y_{K+m}), g^{K}(z_{K+m})\\
& \leq\gamma^{sK} d(g^{(s+1)K}(y_{K+m}), g^{(s+1)K}(z_{K+m}))\\
&\leq\gamma^{sK}\beta
\end{align*}
for all $s\geq 1$, so that $y_m=z_m$.  

For the converse, suppose $y_m=z_m$ for $m=0, \cdots, K-1$ and $\hat{d}(\mathbf{y}, \mathbf{z})\leq\epsilon$.  
For each $m=0, \cdots, K-1$ we have $\hat{g}^{-m}(\mathbf{y})_0=y_m=z_m=\hat{g}^{-m}(\mathbf{z})_0$.  It follows by Observation \ref{d'} that 
$$d'(\hat{g}^{-m+1}(\mathbf{y}), \hat{g}^{-m+1}(\mathbf{z})) =  \gamma d'(\hat{g}^{-m}(\mathbf{y}), \hat{g}^{-m}(\mathbf{z})),$$
and hence 
\[ \hat{d}(\hat{g}(\mathbf{y}), \hat{g}(\mathbf{z})) =  \gamma \hat{d}(\mathbf{y}, \mathbf{z}).\]

That is, $y_m=z_m$ for $m=0, \cdots, K-1$ implies $\hat{d}(\hat{g}(\mathbf{y}), \hat{g}(\mathbf{z})) =  \gamma \hat{d}(\mathbf{y}, \mathbf{z})$.  Let us apply this result to $\hat{g}^n(\mathbf{y})$ and $\hat{g}^n(\mathbf{z})$, where $n\geq 0$.
We have 
$$\hat{g}^n(\mathbf{y})_m=g^n(y_m)=g^n(z_m)=\hat{g}^n(\mathbf{z})_m$$ 
for $m=0, \cdots, K-1$, hence
$$\hat{d}(\hat{g}^{n+1}(\mathbf{y}), \hat{g}^{n+1}(\mathbf{z})) =  \gamma \hat{d}(\hat{g}^n(\mathbf{y}), \hat{g}^n(\mathbf{z})).$$
It follows that 
$$\hat{d}(\hat{g}^n(\mathbf{y}), \hat{g}^n(\mathbf{z})) =  \gamma^n \hat{d}(\mathbf{y}, \mathbf{z})\leq \epsilon$$
for all  $n\geq 0$.  
\end{proof} \medbreak

The following property follows easily from the proof of Lemma \ref{x0=y0}.  This is part (1) of Definition \ref{Smale}.

\begin{cor}\label{stableset}
If $\mathbf{y}, \mathbf{z}\in \hat{Y}^s(\mathbf{x}, \epsilon_{\hat{Y}}')$, then 
$\hat{d}(\hat{g}(\mathbf{y}), \hat{g}(\mathbf{z}))\leq \gamma\hat{d}(\mathbf{y}, \mathbf{z}).$
\end{cor}

Now let us consider the sets $\hat{Y}^u(\mathbf{x}, \epsilon)$.  We observe that the following lemma does not hold if we replace $\hat{d}$ with $d'$;  this is in fact the reason for our use of $\hat{d}$.  

\begin{lemma}\label{ulemma} 
For any $0<\epsilon \leq \epsilon_{\hat{Y}}'$, $\mathbf{y}\in \hat{Y}^u(\mathbf{z}, \epsilon)$ if and only if $d(y_n, z_n)\leq \epsilon$ for every $n\geq 0$ and $\hat{d}(\mathbf{y}, \mathbf{z})\leq\epsilon$.
\end{lemma}

\begin{proof} 
Let $0<\epsilon \leq \epsilon_{\hat{Y}}'$.  

If $\mathbf{y}\in \hat{Y}^u(\mathbf{z}, \epsilon)$ then 
\begin{align*}
d(y_n,z_n)&=d(\hat{g}^{-n}(\mathbf{y})_0, \hat{g}^{-n}(\mathbf{z})_0)\\
& \leq d'(\hat{g}^{-n}(\mathbf{y}), \hat{g}^{-n}(\mathbf{z}))\\
& \leq \hat{d}(\hat{g}^{-n}(\mathbf{y}), \hat{g}^{-n}(\mathbf{z}))\\
&\leq \epsilon
\end{align*}
for all $n\geq 0$.  

Conversely, suppose $d(y_n, z_n)\leq\epsilon$ for all $n\geq 0$ and $\hat{d}(\mathbf{y}, \mathbf{z})\leq \epsilon$.
Since $\epsilon\leq \epsilon_{\hat{Y}}'<\beta$, we can apply Axiom 1 to get $d(g^K(y_n), g^K(z_n))\leq \gamma^K d(g^{2K}(y_n), g^{2K}(z_n))$ for all $n\geq 0$.  Hence
\begin{align*}
d'(\hat{g}^{-K}(\mathbf{y}), \hat{g}^{-K}(\mathbf{z}))& = \mathrm{sup}_{n\geq 0} \lbrace \gamma^n d(y_{K+n},z_{K+n}) \rbrace\\
& =\mathrm{sup}_{n\geq 0} \lbrace \gamma^n d(g^K(y_{2K+n}),g^K(z_{2K+n})) \rbrace\\
& \leq \gamma^K \mathrm{sup}_{n\geq 0} \lbrace \gamma^n d(g^{2K}(y_{2K+n}),g^{2K}(z_{2K+n})) \rbrace\\
& =\gamma^K\mathrm{sup}_{n\geq 0} \lbrace \gamma^n d(y_{n},z_{n}) \rbrace\\
&=\gamma^K d'(\mathbf{y}, \mathbf{z}), 
\end{align*}
which gives
\begin{align*}
\hat{d}(\hat{g}^{-1}(\mathbf{y}), \hat{g}^{-1}(\mathbf{z}))& = \sum_{m=0}^{K-1}\gamma^{-m}d'(\hat{g}^{-m-1}(\mathbf{y}), \hat{g}^{-m-1}(\mathbf{z}))\\
& \leq \gamma^{-(K-1)}\gamma^K d'(\mathbf{y}, \mathbf{z}) + \sum_{m=0}^{K-2}\gamma^{-m}d'(\hat{g}^{-m-1}(\mathbf{y}), \hat{g}^{-m-1}(\mathbf{z}))\\
& =\gamma \left(d'(\mathbf{y}, \mathbf{z}) +\sum_{m=1}^{K-1}\gamma^{-m}d'(\hat{g}^{-m}(\mathbf{y}), \hat{g}^{-m}(\mathbf{z}))\right)\\
&=\gamma \hat{d}(\mathbf{y}, \mathbf{z}).  
\end{align*}

We have shown that $d(y_n, z_n)\leq\beta$ for all $n\geq 0$ implies that 
\[ \hat{d}(\hat{g}^{-1}(\mathbf{y}), \hat{g}^{-1}(\mathbf{z})) \leq \gamma \hat{d}(\mathbf{y}, \mathbf{z}).\]
Let us apply this result to $\hat{g}^{-s}(\mathbf{y})$ and $\hat{g}^{-s}(\mathbf{z})$, where $s\geq 0$.  
We have 
$$d(\hat{g}^{-s}(\mathbf{y})_n, \hat{g}^{-s}(\mathbf{z})_n)=d(y_{n+s}, z_{n+s})\leq\epsilon_{\hat{Y}}'$$ for all $n\geq 0$.  It follows that 
$$\hat{d}(\hat{g}^{-s-1}(\mathbf{y}), \hat{g}^{-s-1}(\mathbf{z}))\leq \gamma \hat{d}(\hat{g}^{-s}(\mathbf{y}), \hat{g}^{-s}(\mathbf{z})),$$
and this is for any $s\geq 0$.  Therefore
$$\hat{d}(\hat{g}^{-n}(\mathbf{y}), \hat{g}^{-n}(\mathbf{z}))\leq \gamma^n \hat{d}(\mathbf{y}, \mathbf{z})\leq \epsilon$$
for every $n\geq 0$.
\end{proof} \medbreak

The following property follows easily from the proof of Lemma \ref{ulemma}.  This is part (2) of Definition \ref{Smale}.

\begin{cor}\label{expanding}
If $\mathbf{y}, \mathbf{z} \in \hat{Y}^u(\mathbf{x}, \epsilon_{\hat{Y}}')$, then 
$\hat{d}(\hat{g}^{-1}(\mathbf{y}), \hat{g}^{-1}(\mathbf{z}))\leq \gamma\hat{d}(\mathbf{y}, \mathbf{z}).$
\end{cor}

\bigskip

Choose
$$0<\epsilon_{\hat{Y}}''\leq  \frac{1}{2}\epsilon_{\hat{Y}}'$$ 
such that $\hat{d}(\mathbf{x}, \mathbf{y})\leq \epsilon_{\hat{Y}}''$ implies $\hat{d}(\hat{g}(\mathbf{x}), \hat{g}(\mathbf{y}))\leq \epsilon_{\hat{Y}}'$.  
Then choose
$$0<\epsilon_{\hat{Y}}\leq \frac{1}{2K}\gamma^K\epsilon_{\hat{Y}}''$$ 
such that
$d(x, y)\leq \epsilon_{\hat{Y}}$ implies $d(g^n(x), g^n(y))\leq \frac{1}{2K}\gamma^{K-1}\epsilon_{\hat{Y}}''$ for $n=K, \cdots, 2K-1$.

\begin{lemma}\label{singleton}
If $\hat{d}(\mathbf{x}, \mathbf{y})\leq \epsilon_{\hat{Y}}$ then $\hat{Y}^s(\mathbf{x}, \epsilon_{\hat{Y}}')\cap \hat{Y}^u(\mathbf{y}, \epsilon_{\hat{Y}}')$ is a singleton.  
\end{lemma}

\begin{proof}
Let $\hat{d}(\mathbf{x}, \mathbf{y})\leq \epsilon_{\hat{Y}}$. Notice that we have
\begin{align*}
\gamma^{-(K-1)}(\gamma^{K} d(x_{2K-1},y_{2K-1}) ) & \leq \gamma^{-(K-1)}d'(\hat{g}^{-(K-1)}(\mathbf{x}), \hat{g}^{-(K-1)}(\mathbf{y})) \\
& \leq \hat{d}(\mathbf{x}, \mathbf{y})\\
& \leq \epsilon_{\hat{Y}}.
\end{align*}
That is, 
\begin{equation}\label{inball}
d(x_{2K-1},y_{2K-1}) \leq \gamma^{-1}\epsilon_{\hat{Y}}<\beta.
\end{equation}

Let us define a point $\mathbf{z}$ by defining $z_{sK}, \cdots, z_{(s+1)K-1)}$ inductively on $s$.  Let  $z_m=x_m$ for $m=0, \cdots K-1$.  
By \eqref{inball} and Axiom 2, we have 
\begin{align*}
z_{K-1} & = x_{K-1}\\
&= g^K(x_{2K-1}) \\
&  \in g^K(B(y_{2K-1}, \gamma^{-1}\epsilon_{\hat{Y}})) \\
& =g^K(B(g^K(y_{3K-1}), \gamma^{-1}\epsilon_{\hat{Y}})) \\
& \subseteq g^{2K}(B(y_{3K-1},\epsilon_{\hat{Y}})),
\end{align*}
so $z_{K-1}=g^{2K}(u_{3K-1})$ for some 
\begin{equation}\label{1st}
u_{3K-1}\in B(y_{3K-1},\epsilon_{\hat{Y}}).
\end{equation}  
Define
\begin{align*}
z_{2K-1} & = g^K(u_{3K-1}) \\
z_{2K-2} & = g(z_{2K-1})=g^{K+1}(u_{3K-1})  \\
& \;\; \vdots \\
z_{K}& = g(z_{K+1})=g^{2K-1}(u_{3K-1}) .
\end{align*}
Observe that we have $g(z_K)=g^{2K}(u_{3K-1})=z_{K-1}$.  

We then use \eqref{1st} and Axiom 2 to get 
\[u_{4K-1}\in B(y_{4K-1}, \epsilon_{\hat{Y}})\]
such that $g^K(u_{3K-1})=g^{2K}(u_{4K-1})$.  We use $u_{4K-1}$ to define $z_{2K},\cdots, z_{3K-1}$; and so on.  
Our construction ensures that $\mathbf{z}\equiv (z_0, z_1, \cdots) \in \hat{Y}$.  

Using Lemmas \ref{x0=y0} and \ref{expanding}, it can be shown that 
\[\mathbf{z}\in \hat{Y}^s(\mathbf{x}, \epsilon_{\hat{Y}}')\cap \hat{Y}^u(\mathbf{y}, \epsilon_{\hat{Y}}'). \]
This is a technical, but not difficult, proof and is hence omitted.  

Let us show that $\mathbf{z}$ is the only point in $\hat{Y}^s(\mathbf{x}, \epsilon_{\hat{Y}}')\cap \hat{Y}^u(\mathbf{y}, \epsilon_{\hat{Y}}')$.  Suppose 
$$\mathbf{v}\in \hat{Y}^s(\mathbf{x}, \epsilon_{\hat{Y}}')\cap \hat{Y}^u(\mathbf{y}, \epsilon_{\hat{Y}}').$$ 
Since $\mathbf{v}, \mathbf{z}\in \hat{Y}^s(\mathbf{x}, \epsilon_{\hat{Y}}')$, we have by Lemma \ref{x0=y0} that  $v_m=x_m=z_m$ for \linebreak $m=0,\cdots, K-1$.  And by Lemma \ref{ulemma}, $\mathbf{v}, \mathbf{z}\in \hat{Y}^u(\mathbf{y},\epsilon_{\hat{Y}}')$ implies 
\begin{equation}\label{vnzn}
d(v_n, z_n)\leq d(v_n,y_n)+d(y_n,z_n)\leq 2\epsilon_{\hat{Y}}'\leq \beta
\end{equation}
for all $n\geq 0$.  
We will complete the proof by induction, by showing that $v_m= \nolinebreak z_m$ \linebreak implies $v_{m+K}=z_{m+K}$.  So suppose that $v_m=z_m$.  We have $d(v_{m+2K}, z_{m+2K})\leq\beta$ from \eqref{vnzn}, and we have assumed that $d(g^{2K}(v_{m+2K}),g^{2K}(z_{m+2K}))=d(v_m,z_m)=\nolinebreak 0$.  It follows by Axiom 1 that 
$v_{m+K}=g^K(v_{m+2K})=g^K(z_{m+2K})=z_{m+K}$.
\end{proof} \medbreak

For points $\mathbf{x},\mathbf{y} \in \hat{Y}$ with $\hat{d}(\mathbf{x},\mathbf{y})\leq\epsilon_{\hat{Y}}$, the bracket $[\mathbf{x},\mathbf{y}]$ is now defined as the unique point in the intersection $\hat{Y}^s(\mathbf{x}, \epsilon_{\hat{Y}}')\cap \hat{Y}^u(\mathbf{y}, \epsilon_{\hat{Y}}')$.  This completes the proof that Axioms 1 and 2 imply that $(\hat{Y}, \hat{d}, \hat{g})$ is a Smale space.  Now we will show that this Smale space has totally disconnected local stable sets.  We will need the following lemma.

\begin{lemma}\label{finite}
Axiom 1 implies that $g$ is finite-to-one.
\end{lemma}

\begin{proof}
Suppose that $Y$ contains an infinite sequence $(y_n)$ of distinct points all having the same image under $g$.
As $g$ is onto, so is $g^K$.  For each $n$, pick $z_n$ with $g^K(z_n) = y_n$. Then $(z_n)$
must have an accumulation point, so we may find $z_m$ and $z_n$ with $m\neq n$ and $d(z_m, z_n)\leq \beta$. So we have $g^{2K}(z_m)= g^{2K}(z_n)$, but $g^K(z_m) =y_m$ and $g^K(z_n) = y_n$ are distinct; this contradicts Axiom 1.
\end{proof} \medbreak

\begin{prop}
If $(Y,d,g)$ satisfies Axioms 1 and 2, then the Smale space $(\hat{Y}, \hat{d}, \hat{g})$ has totally disconnected local stable sets.
\end{prop}

\begin{proof}
For $n\geq 0$, denote by $\pi_n: \hat{Y}\rightarrow Y$ the projection map $\pi_n(y_0, y_1, y_2, \cdots ) = \nolinebreak y_n$.  Choose $\mathbf{y}\in \hat{Y}$. 
By Lemma \ref{x0=y0}, every point in $\hat{Y}^s(\mathbf{y}, \epsilon_{\hat{Y}})$ has the same first coordinate, $y_0$.  Therefore, for any $n\geq 0$, the set $\pi_n(\hat{Y}^s(\mathbf{y}, \epsilon_{\hat{Y}}))\subseteq g^{-n}\lbrace y_0 \rbrace$ is finite by Lemma \ref{finite}.  
So the $\pi_n$ preimage of any point in this finite
set is clopen in $\hat{Y}^s(\mathbf{y}, \epsilon_{\hat{Y}})$. As a result, for any two distint points in $\hat{Y}^s(\mathbf{y}, \epsilon_{\hat{Y}})$, we can find a clopen set containing one but not the other.
\end{proof} \medbreak

A Smale space is said to be \textit{irreducible} if it is non-wandering and has a dense orbit.  It is well known that $(Y,g)$ has these properties if and only if its stationary inverse limit $(\hat{Y},\hat{g})$ does.  For a proof of these facts, see \cite{Aoki}.

\section{Proof of Theorem B}\label{thmBpf}

Let $(X,d,f)$ be an irreducible Smale space whose local stable sets are totally disconnected, with constants $\epsilon_X>0$, $\epsilon_X'>0$, and $0<\lambda<1$ as in Definition \ref{Smale}.

We will use a Markov partition with a special property to define an equivalence relation, $\sim$, on $X$.  We then define a metric, $\delta$, and a mapping, $\alpha$, on the quotient  $X/_\sim$.  We will show that $(X/_\sim, \delta, \alpha)$ satisfies Axioms 1 and 2, and that
$$\underleftarrow{\mathrm{lim}}\;X/_\sim \stackrel{\alpha}{\longleftarrow} X/_\sim \stackrel{\alpha}{\longleftarrow}  \cdots,$$
together with the map $\hat{\alpha}$ and metric $\hat{\delta}$, is conjugate to $(X,d,f)$. 

The relation $\sim$ has the effect of collapsing each Markov partition rectangle to a single unstable set (see Figure \ref{sim}).  These unstable sets may intersect on the boundaries, making the definition an appropriate metric on $X/_\sim$ rather difficult.  The other aspects of our construction of the inverse limit space are quite intuitive.  

\begin{figure}[h]\label{sim}
\psset{xunit=1.5cm,yunit=1.5cm}
\begin{center}
\begin{pspicture}(0,0)(5.3,2)
\pspolygon(0,0)(1.5,0)(1.5,1)(0,1)
\pspolygon(.7,1)(2.7,1)(2.7,1.8)(.7,1.8)
\pspolygon(1.8,0)(3,0)(3,1)(1.8,1)
\rput(4,1){$\longrightarrow$}
\rput(4,1.2){$\sim$}
\psline(5,1.8)(5,1)
\psline(5,1)(4.7,0)
\psline(5,1)(5.3,0)
\psdot(5,1)
\end{pspicture}
\end{center}
\caption{The equivalence relation $\sim$}
\end{figure}
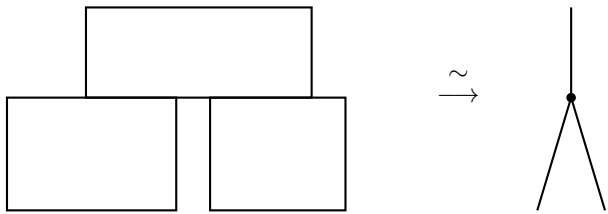

\subsection{Construction of the Quotient Space}

A non-empty set $R\subseteq X$ is called a \textit{rectangle} if 
$R=\overline{\mathrm{Int}(R)}$ and $[x,y]\in R$ whenever $x,y\in R$.  The second condition tells us that we must have  diam$(R)\leq \epsilon_X$.

For a rectangle $R$ and $x\in R$, we will denote $X^s(x,R)=X^s(x,\epsilon_X)\cap R$
and
$X^u(x,R)=X^u(x,\epsilon_X)\cap R.$

A finite cover $\mathcal{P}=\lbrace R_1, R_2, \cdots, R_n \rbrace$ of $X$ by rectangles is a \textit{Markov partition} provided that 
\begin{enumerate}
\item $\mathrm{Int}(R_i)\cap \mathrm{Int}(R_j)=\varnothing$ for $i\neq j$, and 
\item $f(X^s(x, R_i))\subseteq X^s(f(x),  R_j)$ and $f^{-1}(X^u(f(x),  R_j))\subseteq X^u(x, R_i)$
whenever $x\in \mathrm{Int}(R_i)\cap f^{-1}(\mathrm{Int}(R_j))$.  This is called the ``Markov property".
\end{enumerate}

Bowen \cite{Bowen} proved that all irreducible Smale spaces have Markov partitions.  But a generic Markov partition is not sufficient in our case.  We need a Markov partition where each rectangle is clopen in the stable direction.  Our proof of the existence of such a partition relies heavily on a number of properties of $s$-resolving factor maps; a factor map between two Smale spaces is $s$-resolving if it is  injective on the local stable sets of its domain. Alternatively, Proposition \ref{clopenMP} could also be proved by following Bowen's construction for a generic Markov partition and making some necessary adjustments along the way.  

\begin{prop}[Putnam \cite{Putnamnotes}]\label{unstab}
Let $\pi:(X,f)\rightarrow (Y,g)$ be an $s$-resolving factor map between irreducible Smale spaces.  Then
\begin{enumerate}
\item $\pi$ is a homeomorphism on the local stable sets $X^s(x,\epsilon)$, 
\item $\pi$ is finite-to-one, and
\item for every  point $y_0$ in $Y$ with a dense forward orbit
 we have 
$$ \# \pi^{-1}\{ y_0 \}=\deg(\pi)\equiv \mathrm{min}\lbrace \#\pi^{-1}\lbrace y \rbrace \; | \; y\in Y \rbrace.$$ 
\end{enumerate}
Furthermore, there exists $\epsilon_\pi>0$ such that
\begin{enumerate}[resume]
\item  for all $x_1,x_2\in X$ with $d_X(x_1, x_2)\leq \epsilon_\pi$, we have $[x_1, x_2]$ and $[\pi(x_1), \pi(x_2)]$ both defined and  
\[ [\pi(x_1), \pi(x_2)]=\pi([x_1,x_2]),\]
\item if $\pi(x_1)\in Y^u(\pi(x_2), \epsilon_Y)$ and $d(x_1, x_2)\leq \epsilon_\pi$, then $x_1\in X^u(x_2, \epsilon_\pi),$ and
\item if $x, x'\in X$ with $\pi(x)=\pi(x')$ and
$\liminf_{n\to \infty}  d(f^n(x), f^n(x')) < \epsilon_\pi$, then \linebreak $x= x'$. 
\end{enumerate}
\end{prop}

\begin{prop}\label{clopenMP}
Let $(X,f)$ be an irreducible Smale space such that $X^s(x,\epsilon)$ is totally disconnected for every $x\in X$ and $0< \epsilon \leq \epsilon_X$.  Then there exists a Markov partition, $\mathcal{P}$, for $(X,f)$ such that if $x\in R \in \mathcal{P}$, then $X^s(x,R)$ is clopen in $X^s(x,\epsilon_X)$.  
\end{prop}

\begin{proof}
By Corollary 1.3 of \cite{Putnam}, there exists an irreducible shift of finite type, $\Sigma$, and an  $s$-resolving factor map $\pi: \Sigma \rightarrow X$.  The metric on $\Sigma$ is the common one given by $d_\Sigma(\mathbf{s},\mathbf{t})=\sum_{n\in \mathbb{Z}} 2^{-|n|}\chi(s_n,t_n)$, where
\[ \chi(s_n, t_n)=\left\lbrace\begin{array}{ll}
0 & \mathrm{if} \; s_n=t_n\\
1 & \mathrm{if} \; s_n \neq t_n
\end{array}\right. .\]  And the homeomorphism on $\Sigma$ is the usual left shift map, which we denote by $S$.  

Let $\epsilon_\pi>0$ be as in Proposition \ref{unstab}.  Choose $N\in \mathbb{N}$ such that 
$$\sum_{|n|>N}2^{-n}<\epsilon_\pi.$$
Let $P_{2N+1}$ be the set of all paths of length $2N+1$ which appear in elements of $\Sigma$.

For $w\in P_{2N+1}$, let $R_w=\lbrace \mathbf{a}\in \Sigma \; | \; a_{-N}\cdots a_N =w \rbrace.$ Then $R_w$ is a clopen rectangle in $\Sigma$ with diameter less than $\epsilon_\pi$, and $\mathcal{P}=\lbrace R_w \; | \; w\in P_{2N+1}\rbrace$ is a Markov partition for $\Sigma$.  

Since each $R_w\in \mathcal{P}$ is compact in $\Sigma$, it follows that $\pi(R_w)$ is compact in $X$, and hence closed.
Moreover, since $\pi$ is $s$-resolving and each $R_w$ is clopen, it follows that each $\pi (R_w)$ is clopen in the stable direction.  Let's show that $[x,y]\in \pi(R_w)$ whenever $x,y\in \pi(R_w)$.  Suppose $x=\pi(\mathbf{a})$ and $y=\pi(\mathbf{b})$ for some $\mathbf{a},\mathbf{b}\in R_{w}$.  Since \linebreak diam$(R_{w})\leq \epsilon_\pi$, it follows from Proposition \ref{unstab} (4) that we must have \linebreak $[x,y]=[\pi(\mathbf{a}),\pi(\mathbf{b})]=\pi([\mathbf{a},\mathbf{b}])\in \pi(R_{w})$.

By Proposition \ref{unstab} (2), $\pi$  is finite-to-one; let $d=\mathrm{deg}(\pi)$.
We will show that a subset of 
\[  \lbrace \pi (R_{w_1})\cap \pi (R_{w_2}) \cap \cdots \cap \pi (R_{w_d}) \;| \;
 R_{w_1}, R_{w_2}, \cdots, R_{w_d} \in  \mathcal{P}\; \mathrm{distinct} \rbrace\]
is a Markov partition for $(X,f)$.  

Let us define a map $n:X \rightarrow \mathbb{N}$ by 
\[n(x)= \#\lbrace  R_w\in \mathcal{P}\;| \; x\in \pi(R_w)\rbrace.\]
Since the $R_w$ are disjoint, it follows that 
\begin{equation}\label{lbound}
n(x) \leq \#\pi^{-1}\lbrace x \rbrace
\end{equation}
for all $x\in X$.  

We have the following estimate of continuity of $n$.  Suppose we have  a convergent sequence $x_{k}$ with limit point $x$. Since each $x_{k}$ lies in $n(x_k)$ elements of the finite set $\lbrace \pi(R_w)\; | \; R_w\in \mathcal{P}\}$,  we may pass to a subsequence where every term is contained in the same $\pi(R_w)$'s. Since they are closed, $x$ also lies in these $\pi(R_w)$'s. Hence 
\begin{equation}\label{semicont}
n(x) \geq \limsup n(x_{k}).
\end{equation}

Let us show that $n(x)\geq d$ for all $x\in X$, and that equality holds if $x$ has a dense forward orbit.  Let $x$ be any point in $X$ and let $x_0\in X$ have a dense forward orbit (such a point exists since $(X,f)$ is irreducible).  By Theorem \ref{unstab} (3), $\#\pi^{-1}\lbrace x_0 \rbrace = d$; let $\pi^{-1}\{ x_0 \} = \{ \mathbf{a}_{1}, \mathbf{a}_{2}, \ldots, \mathbf{a}_d \}$.
Choose a sequence of positive integers so that $f^{m_{k}}(x_0)$ converges to $x$. 
Pass to a subsequence where $S^{m_{k}}(\mathbf{a}_{j})$ converges, for each $ 1 \leq j \leq d$.
If two of the limit points (for different values of $j$) are in the same rectangle, then they are within $\epsilon_{\pi}$ of each
other.  So by Theorem \ref{unstab} (6), these two $\mathbf{a}_{j}$'s are equal. Since this isn't the case, we see
that no two limit points of the sequences can be in the same rectangle, but they all clearly lie in  $\pi^{-1}\{ x \}$.
As a result, $n(x)\geq d$.  It follows from \eqref{lbound}  that $n(x_0)=d$.  

That is, $n^{-1}\{ d \}$ is non-empty and $n^{-1}\{ k \}$ is empty for $k < d$. From \eqref{semicont}
we also see that $n^{-1}\{ d+1, d+2, \ldots \}$ is closed and so $n^{-1}\{ d \}$ is open.  We claim it is also dense. 
But that follows from the fact that it contains all points with a dense forward orbit. One of them is enough, since each point in its forward orbit also has a dense forward orbit.

Let 
\begin{align*}
\mathcal{S}=\lbrace \lbrace R_{w_1}, R_{w_2}, \cdots, R_{w_d}\rbrace\subseteq \mathcal{P}\; | \; & \exists \; x\in n^{-1}\{ d \} \; \mathrm{with} \\ 
&x\in \pi(R_{w}) \Leftrightarrow w\in \lbrace w_1, w_2, \cdots, w_d\rbrace \rbrace.
\end{align*}
We will show that 
\[  \mathcal{R}=\lbrace \pi (R_{w_1})\cap \pi (R_{w_2}) \cap \cdots \cap \pi (R_{w_d}) \;| \;
 \lbrace R_{w_1}, R_{w_2}, \cdots, R_{w_d}\rbrace \in \mathcal{S} \rbrace\]
 is a Markov partition for $(X,f)$.  We already observed above that each $\pi(R_w)$ is clopen in the stable direction; it is clear that a finite intersection of these sets would have the same property.  
 
First we need to know that the elements of $\mathcal{R}$ are rectangles.   
That they have dense interiors follows from the fact that $n^{-1}\{ d \}$ is open and dense in $X$.  Moreover, we observed above that for any $R_w\in \mathcal{P}$, we have $[x,y]\in \pi(R_w)$ whenever $x,y\in \pi(R_w)$.  

That $\mathcal{R}$ covers $X$ and that the elements of $\mathcal{R}$ have disjoint interiors also follows from the fact that $n^{-1}\{ d \}$ is open and dense in $X$.

So it remains to prove that $\mathcal{R}$ satisfies the Markov property.  It suffices to prove this for the set of points in $X$ with dense forward orbits, since these points (and their orbits) are clearly contained in the interiors of elements of $\mathcal{R}$.  

Let $x\in \mathrm{Int}(\bigcap_{i=1}^d\pi(R_{w_i}))\cap f^{-1}(\mathrm{Int}(\bigcap_{i=1}^d\pi(R_{v_i})))$, where  $\bigcap_{i=1}^d\pi(R_{w_i})$ and \linebreak $\bigcap_{i=1}^d\pi(R_{v_i})$ are elements of $\mathcal{R}$.  Since $\pi^{-1}\{x\}=\{ \mathbf{a}_1, \cdots, \mathbf{a}_d \}$ and $n(x)=d$, it follows that for each $\mathbf{a}_k$ there are $1\leq i,j \leq d$ such that $\mathbf{a}_i\in  R_{w_i} \cap S^{-1}(R_{v_j})$.  Therefore 
$S(\Sigma^s(\mathbf{a}_k, R_{w_i}))\subseteq \Sigma^s(S(\mathbf{a}_k),  R_{v_j})$ and $S^{-1}(\Sigma^u(S(\mathbf{a}_k),  R_{v_j}))\subseteq \Sigma^u(\mathbf{a}_k, R_{w_i})$.  By Proposition \ref{unstab} (1), $\pi$ is a homeomorphism on the local stable sets, so that 
\[f(X^s(x, \pi(R_{w_i})))=f(\pi(\Sigma^s(\mathbf{a}_k, R_{w_i}))) =\pi( S(\Sigma^s(\mathbf{a}_k, R_{w_i})))\subseteq \pi( R_{v_j}).\]
And by Proposition \ref{unstab} (5), we also have 
\begin{align*}
f^{-1}(X^u(f(x), \pi(R_{w_j}))) & \subseteq f^{-1}(\pi(\Sigma^u(S(\mathbf{a}_k), R_{w_j}))) \\
& =\pi( S^{-1}(\Sigma^u(S(\mathbf{a}_k), R_{w_j})))\\
& \subseteq \pi( R_{v_i}).
\end{align*}
Since $f(X^s(x, \epsilon_X))\subseteq X^s(f(x), \epsilon_X)$ and $f^{-1}(X^u(f(x), \epsilon_X))\subseteq X^u(x, \epsilon_X)$ hold trivially, we are done.  
\end{proof}

Let $\mathcal{P}=\lbrace R_1, \cdots, R_M\rbrace$ be a Markov partition for $(X,f)$ as in Proposition \ref{clopenMP}; that is, for every $x\in R_i\in  \mathcal{P}$,  $X^s(x, R_i)$ is clopen in $X^s(x,\epsilon_X)$.  %Moreover, the diameters of the rectangles in $\mathcal{P}$ can be chosen arbitrarily small.  So we will assume that for each $R\in \mathcal{P}$, diam$(R)\leq \frac{1}{2}\epsilon_X'$, where $0<\epsilon_X'\leq \frac{\epsilon_X}{3}$ is the constant from Proposition \ref{epsilonX'}.  

%It is an easy consequence of Proposition \ref{epsilonX'} that for any $x\in \mathrm{Int}(R_i)$ we have $X^s(x, R_i)\subseteq \mathrm{Int}(R_i)$.  

We define a relation $\thickapprox$ on $X$ as follows:
\begin{center}
$x \thickapprox y$ if and only if $x,y\in R_i$ for some $R_i\in\mathcal{P}$ and $x\in X^s(y,\epsilon_X)$.  
\end{center}
We observe that this relation is reflexive and symmetric, but not transitive.  Let $\sim$ be the transitive closure of $\thickapprox$, with equivalence classes denoted $[[\cdot ]]$.  That is, $x\sim y$ if and only if there are $x_1, x_2, \cdots, x_n\in X$ such that $x \thickapprox x_1 \thickapprox \cdots \thickapprox x_n \thickapprox y$.  

We observe that $[[x]]\subseteq X^s(x)\equiv  \lbrace y\in X \; | \; \lim_{n\rightarrow \infty}d(f^n(x),f^n(y))=0\rbrace$.  Moreover, if $x \in \mathrm{Int}R_i$ then $x\thickapprox X^s(x,R_i) \subseteq \mathrm{Int}R_i$, and $\mathrm{Int}R_i\cap R_j=\emptyset$ if $i\neq j$, so that $[[x]]=X^s(x,R_i)$.  
That is, the equivalence classes, $[[\cdot]]$, are larger on the boundaries of the Markov partition rectangles than they are on the interiors; by ``larger" we mean that they intersect more rectangles.  As a result we have the following intuitive sense of lower semi-continuity on the local unstable sets:  for a sequence $(x_n)$ converging to $x$, $[[x]]$ can be larger than $[[x_n]]$ but not smaller.  

To define our metric on $X/_\sim$, we will enlarge the equivalence classes $[[\cdot]]$ near the boundaries of the Markov partition rectangles, and then define paths using these enlarged classes.  The distance between $[[x]]$ and $[[y]]$ will be defined to be the length of the shortest path between $[[x]]$ and $[[y]]$. 
A distinctive feature of our paths is that they are concatenations of very short moves within local stable or unstable sets, where the moves in the stable sets do not contribute to the length of the path.  
A variation of our metric appears in \cite{Putnam}.

The bracket map is the natural candidate for a tool to enlarge the $[[\cdot]]$; however this map is defined only on small balls.  So we develop a new tool, $\langle \cdot, \cdot \rangle$, for this purpose.    

It is a well-known fact that if a metric space $A$ is compact and $C \subseteq A$ is a clopen subset, then there exists $\epsilon > 0$ such that $B(C, \epsilon) \subseteq C$.  Since we know that the rectangles are clopen in the stable direction, we would like a uniform constant satisfying this property.  

\begin{lemma}\label{epsilon_0}
There exists $0<\epsilon_0\leq \frac{1}{3}\epsilon_X$ such that if $x\in R_i\in\mathcal{P}$ then 
$$X^s(x,\epsilon_0)\subseteq R_i.$$
\end{lemma}

\begin{proof}
Choose $x_i\in R_i\in\mathcal{P}$.  Since $X^s(x_i, R_i)$ is clopen in the compact set $X^s(x_i,\epsilon_X)$, there exists $0<\epsilon_i\leq \frac{1}{3}\epsilon_X$ such that 
\begin{equation}\label{compactball}
B(X^s(x_i, R_i),\epsilon_i)\cap X^s(x_i, \epsilon_X)\subseteq X^s(x_i, R_i)
\end{equation}
The collection $\lbrace X^s(y,\frac{1}{2}\epsilon_i) \; | \; y\in X^s(x_i, R_i) \rbrace$ covers $X^s(x_i, R_i)$, so there is a finite subcover, with centers $y_1, \cdots, y_n$.  

By the uniform continuity of $[ \cdot, \cdot ]$, there exists $0< \eta_i \leq \frac{1}{2}\epsilon_i$ such that if $d(a, b)\leq \nolinebreak \eta_i$ then $d([c,a],[c,b])\leq \frac{1}{2}\epsilon_i$ for any $c\in B(a,\epsilon_X)\cap B(b,\epsilon_X)$.  We will show that 
$$X^s(x,\eta_i)\subseteq R_i$$
for any $x\in R_i$.  

Choose $x\in R_i$ and $y\in X^s(x,\eta_i)$.  Since 
$R_i=[X^u(x_i, R_i),X^s(x_i, R_i)]$, it follows that 
$x=[u,s]$ for some $u\in X^u(x_i, R_i)$ and $s\in X^s(x_i, R_i)$.  And $s\in X^s(y_j,\frac{1}{2}\epsilon_i)$ for some $1\leq j \leq n$.  Moreover, since $d(x,y)\leq\eta_i$, it follows that we have \linebreak $d(s,[y_j,y])=d([y_j,x],[y_j,y])\leq \frac{1}{2}\epsilon_i.$
Hence
\begin{align*}
d([y_j,y],y_j)& \leq d([y_j,y],s) + d(s,y_j)\\
& \leq \tfrac{1}{2} \epsilon_i+ \tfrac{1}{2} \epsilon_i\\
& = \epsilon_i, 
\end{align*}
so that $[y_j,y]\in X^s(y_j,\epsilon_i)$.  Moreover, since $y_j\in X^s(x_i, R_i)\subseteq X^s(x_i, \frac{1}{2}\epsilon_X)$, it follows that 
$X^s(y_j,\epsilon_i)\subseteq X^s(y_j, \frac{1}{3}\epsilon_X)\subseteq X^s(x_i,\epsilon_X)$. That is, 
$$[y_j,y]\in B(y_j,\epsilon_i)\cap X^s(x_i, \epsilon_X)\subseteq B(X^s(x_i, R_i),\epsilon_i)\cap X^s(x_i, \epsilon_X).$$  It follows from \eqref{compactball} that $[y_j,y]\in R_i$.  By the definition of a rectangle, $x, [y_j,y]\in R_i$ implies
$$y=[x,y] = [x,[y_j,y]] \in R_i.$$

Let $\epsilon_0 = \mathrm{min} \lbrace \eta_i \; | \; i=1, \cdots, M \rbrace$.
\end{proof}

Next, we find a bound on the transitive closure, $\sim$, of the relation $\thickapprox$.  

\begin{lemma}\label{N}
There exists $N\in \mathbb{N}$ such that if $y\sim x$ then there are $y_1,\cdots, y_N$ with 
$$y\thickapprox y_1 \thickapprox \cdots \thickapprox y_N \thickapprox x.$$
\end{lemma}

\begin{proof}
Let $\epsilon_0>0$ be as in Lemma \ref{epsilon_0} and choose $m\in \mathbb{N}$ such that 
$$\lambda^m \epsilon_X \leq\epsilon_0.$$
Then choose $\eta>0$ such that $d(x,y)< \eta$ implies $d(f^{-k}(x), f^{-k}(y))<\epsilon_X$ for all $k=0, \cdots, m$.  
Cover $X$ by $\frac{1}{2}\eta$-balls and extract a finite subcover $\lbrace B_1, \cdots, B_n\rbrace$.  

We claim that $N=2nM-2$ satisfies the conclusion.  

Let $x,y\in X$ with $y\sim x$.  By definition of $\sim$, we know that $y\thickapprox y_1 \thickapprox \cdots \thickapprox y_L \thickapprox x$ for some $y_1, \cdots, y_L\in X$.  Denote $y_0=y$ and $y_{L+1}=x$.  Suppose $L>2nM-2$.  We will show that $y_j\thickapprox y_{j'}$ for some non-consecutive $j$ and $j'$.  

Since $P$ covers $X$, $f^m(y_0)\in R_i$ for some (not necessarily unique) $1\leq i \leq M$.  By definition of $\thickapprox$, we have $y_{j+1}\in X^s(y_j, \epsilon_X)$ for all $j=0, \cdots, L$, so that 
$$f^m(y_{j+1})\in X^s(f^m(y_j), \lambda^m\epsilon_X)\subseteq X^s(f^m(y_j), \epsilon_0).$$
Arguing inductively we see that   
\begin{equation}\label{R_i}
f^m(y_j)\in X^s(f^m(y_0), \epsilon_X) \cap R_i
\end{equation}
for all $j=1, \cdots, L +1$.

Since $L+2 \geq 2nM +1$, at least $2n+1$ of the $y_j$ are in the same $R_{i'}$ for some $1\leq i'\leq M$.  Look at those and apply $f^m$ to all of them.  Since there are at least $2n+1$ of these, at least $3$ of them are in the same $B_r$ for some $1\leq r \leq n$.  Of these 3, choose 2 that have non-consecutive indices.  That is, there are $y_j$ and $y_{j'}$ with $j$ and $j'$ not consecutive, such that $y_j, y_{j'}\in R_{i'}$ and $f^m(y_j),f^m(y_{j'})\in B_r$.  So
$$d( f^m(y_j),f^m(y_{j'}))<\eta,$$
and we have from \eqref{R_i} that $f^m(y_j), f^m(y_{j'})\in X^s(f^m(y_0), \epsilon_X)$, hence
$$f^m(y_j) \in X^s(f^m(y_{j'}),\eta).$$
It follows from our choice of $\eta>0$ that 
$$y_j\in X^s(y_{j'}, \epsilon_X).$$
Therefore $y_j\thickapprox y_{j'}$.  
\end{proof}

Choose $0< \epsilon_0''\leq \epsilon_0$ such that $[x,y]\in X^s(x,\epsilon_0)\cap X^u(y,\epsilon_0)$ if $d(x,y)\leq \epsilon_0''$ (the existence of such a constant follows from the continuity of $[\cdot, \cdot]$).

Let $\epsilon_0'=\mathrm{min}\lbrace\frac{1}{3}\epsilon_X', \frac{1}{3}\epsilon_0''\rbrace$.

We noted that $[[x]]\subseteq X^s(x)$.  It follows that for each $y\in [[x]]$ there exists $n\geq 0$ such that $f^n(y)\in X^s(f^n(x), \epsilon_0')$.  We now show that we can do this in uniform time.  

\begin{cor}\label{K}
There exists $K\in \mathbb{N}$ such that if $y\in [[x]]$ then $f^K(y)\in X^s(f^K(x), \epsilon_0')$. 
\end{cor}

\begin{proof}  Suppose $y\sim x$.  Then by Lemma \ref{N}, there exist $y_1,\cdots, y_N$ with 
$$y\thickapprox y_1 \thickapprox \cdots \thickapprox y_N \thickapprox x.$$
Denote $y_0=y$ and $y_{N+1}=x$.  We have
$$y_i\in X^s(y_{i+1}, \epsilon_X)$$
for all $i=0, \cdots, N$.  Choose $K\in \mathbb{N}$ such that
$$\lambda^K (N+1)\epsilon_X\leq \epsilon_0'.$$
We have 
$$f^K(y_i)\in X^s(f^K(y_{i+1}), \lambda^K\epsilon_X)$$
for all $i=0, \cdots, N$, 
so that
\begin{align*}
f^K(y_0) & \in X^s(f^K(y_1), \lambda^K\epsilon_X) \\
&\subseteq X^s(f^K(y_2), 2\lambda^K\epsilon_X) \\
& \qquad \qquad \qquad \vdots \\
& \subseteq X^s(f^K(y_{N+1}), (N+1)\lambda^K\epsilon_X)\\
& \subseteq X^s(f^K(y_{N+1}), \epsilon_0').
\end{align*}
That is, $f^K(y)\in X^s(f^K(x), \epsilon_0')$.
\end{proof}

It follows quite easily from Corollary \ref{K} that the equivalence classes $[[\cdot]]$ are closed.  

Now, choose 
\begin{equation}\label{eX''}
0< \epsilon_X'' \leq \epsilon_0
\end{equation}
such that $d(x,y)\leq\epsilon_X''$ implies that $d(f^k(x), f^k(y))\leq \epsilon_X$ for each $k=0, \cdots, K$.  Then choose 
\begin{equation}\label{e1}
0<\epsilon_1\leq \epsilon_0
\end{equation}
such that $d(x,y)\leq\epsilon_1$ implies that $d(f^k(x), f^k(y))\leq \epsilon_0'$ for each $k=-K, \cdots, K$, and that $[x,y] \in X^s(x, \epsilon_X'')\cap X^u(y, \epsilon_X'')$ (a constant satisfying the second property exists by the continuity of $[\cdot,\cdot ]$).  

Suppose $d([[y]],[[x]])=\mathrm{inf}\lbrace d(y',x') \; \vert \; y'\in [[y]], \; x'\in [[x]]\rbrace \leq \epsilon_1$.  Then there exist $u\in [[y]]$ and  $v\in [[x]]$ with $d(u,v)\leq \epsilon_1$.  Let $z\in [[y]]$ and $x'\in [[x]]$.  We observe that 
\begin{align}
d(f^K(z),f^K(x')) &\leq d(f^K(z),f^K(u)) + d(f^K(u),f^K(v)) + d(f^K(v),f^K(x')) \notag \\
&\leq 3\epsilon_0' \label{smallenough}\\
& < \epsilon_X, \notag
\end{align}
so that $[f^K(z),f^K(x')]$ is defined. Moreover, since $d(f^K(z),f^K(x'))\leq \epsilon_X'$, it follows that  
\begin{align*}
[f^K(z),f^K(x')] & \in X^s\left(f^K(z), \tfrac{1}{3}\epsilon_X\right)\cap X^u(f^K(x'), \epsilon_X)\\
& \subseteq X^s(f^K(y), \epsilon_X)\cap X^u(f^K(x'), \epsilon_X)\\
& =\lbrace [ f^K(y),f^K(x') ] \rbrace;
\end{align*}
that is, $[f^K(z),f^K(x')]=[f^K(y),f^K(x')]$.

Define 
$$\langle y, [[x]] \rangle = \lbrace f^{-K}[f^K(y),f^K(x')] \; \vert \; x'\in [[x]] \rbrace .$$
We just showed that $\langle y, [[x]] \rangle= \langle z, [[x]] \rangle$.  And we observe that 
\begin{align*}
\langle x, [[x]] \rangle & = \lbrace f^{-K}[f^K(x),f^K(x')] \; \vert \; x'\in [[x]] \rbrace \\
&= \lbrace f^{-K}[f^K(x'),f^K(x')] \; \vert \; x'\in [[x]] \rbrace\\
&= \lbrace x' \; \vert \; x'\in [[x]] \rbrace\\
&=[[x]].
\end{align*}

By the uniform continuity of $f$, $f^{-1}$, and $[\cdot, \cdot]$, for each $0<\epsilon\leq \epsilon_1$ there exists 
$$0<\beta(\epsilon)\leq  \epsilon$$ 
such that $d_2((a,b),(c,d))\leq \beta(\epsilon)$ and $(f^K(a),f^K(b)), (f^K(c),f^K(d))\in \mathrm{Domain}[\cdot, \cdot]$ implies
$$d(f^{-K}[f^K(a),f^K(b)],f^{-K}[f^K(c),f^K(d)])\leq \epsilon.$$ 

\begin{lemma}\label{epsilonx}  
For each $x\in X$, there exists $0< \epsilon_{[[x]]}\leq \beta(\epsilon_1)$ such that \linebreak $d(y,[[x]])\leq \epsilon_{[[x]]}$ implies $[[y]]\subseteq \langle y, [[x]] \rangle$.
\end{lemma}

\begin{proof}
Suppose that for each $n\in \mathbb{N}$ with $\frac{1}{n}\leq \beta(\epsilon_1)$ there exists $y_n\in X$ with $d(y_n,[[x]])\leq \frac{1}{n}$ and $[[y_n]]\nsubseteq \langle y_n, [[x]] \rangle$; i.e. there exists $y_n'\in [[y_n]]$ with $y_n'\notin \nolinebreak \langle y_n, [[x]] \rangle$.  We claim that $d(y_n', [[x]]) \geq \epsilon_1$.

Suppose $d(y_n', [[x]])<\epsilon_1$.  Then there exists $x'\in[[x]]$ such that $d(y_n',x')< \epsilon_1$.  By our choice of $\epsilon_1>0$, it follows that $[x',y_n']\in X^s(x',\epsilon_X'')\subseteq X^s(x',\epsilon_0)$.  By our choice of $\epsilon_0$, this implies $[x',y_n']\thickapprox x'$; that is, $[x',y_n']\in[[x]]$.
Furthermore, we also have $[x',y_n']\in X^u(y_n',\epsilon_X'')$.  Therefore $d(f^k(y_n'),f^k([x',y_n']))\leq \epsilon_X$ for each $0\leq k \leq K$, so that $[f^k(y_n'),f^k([x',y_n'])]$ is defined for each $0\leq k \leq K$.  It follows from Definition \ref{Smale} that  
$$f^K[y_n',[x',y_n']]=[f^K(y_n'),f^K[x',y_n']].$$
Therefore
\begin{align*}
y_n'& = [y_n',[x',y_n']] \\
& = f^{-K}[f^K(y_n'),f^K[x',y_n']] \\
& \in \langle y_n', [[x]] \rangle\\
& = \langle y_n, [[x]] \rangle,
\end{align*}
a contradiction.

Since $X$ is compact, there exists a convergent subsequence $y_{n_k}\rightarrow y$ of $(y_n)$, and a convergent subsequence $y_{n_{k_j}}'\rightarrow y'$ of $(y_{n_k}')$.  So $y\sim y'$.  Moreover, since each $d(y_{n_{k_j}}',[[x]])\geq \epsilon_1$, it follows that $d(y',[[x]])\geq \epsilon_1$.

However, since $[[x]]$ is closed as well and $d(y,[[x]])\leq d(y,y_{n_k})+d(y_{n_k},[[x]]) \rightarrow 0$, it follows that $y\in [[x]]$, and hence $y'\in [[x]]$.  So we have $0=d(y',[[x]])\geq \epsilon_1$.
\end{proof}

For sets $A,B\subseteq X$ and $0<\epsilon\leq \epsilon_X$, let 
$$X^u(A,B,\epsilon)=\lbrace (a,b) \; | \;  a\in A,\; b\in B, \; a\in X^u(b,\epsilon)\rbrace$$
(see Figure \ref{X^u}).
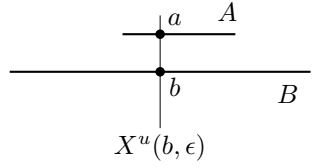
\begin{figure}[h!]
\begin{center}
\begin{pspicture}(0,0)(4,2)
\psline(0,1)(4,1)
\rput(3.7, .7){$B$}
\psline(1.5, 1.5)(3, 1.5)
\rput(2.9, 1.8){$A$}
\psdot(2,1)
\rput(2.2, .8){$b$}
\psdot(2,1.5)
\rput(2.2, 1.7){$a$}
\psline[linewidth=0.1pt](2, .25)(2, 1.75)
\rput(2, 0){$X^u(b, \epsilon)$}
\end{pspicture}
\end{center}
\caption{$(a,b)\in X^u(A, B, \epsilon)$}
\label{X^u}
\end{figure}
 
\noindent We define
\begin{equation*}
d^u(A,B) =\left\lbrace \begin{array}{cl}
\mathrm{sup}\lbrace d(a,b)\; | (a,b)\in X^u(A,B,\epsilon_1)\rbrace & \mathrm{if} \; X^u(A,B,\epsilon_1)\neq \varnothing\\
 \epsilon_1 & \mathrm{otherwise}
\end{array}\right..
\end{equation*}

As a distance function on $X/_\sim$, $d^u$ is clearly symmetric and we will show that it is reflexive, but the triangle inequality fails.  
To prove the reflexivity of $d^u$ on $X/_\sim$, let $(y,z)\in X^u([[x]],[[x]],\epsilon_1)$.  Then by our choice of $K$, 
\begin{equation}\label{uptoK}
f^K(y)\in X^s(f^K(z), \epsilon_0').
\end{equation}
Since we also have $d(y,z)\leq \epsilon_1$, it follows from \eqref{uptoK} and our choice of $\epsilon_1$ that \linebreak $y\in X^s(z,\epsilon_0')$.  So we have 
$$y\in X^s(z,\epsilon_0') \cap X^u(z,\epsilon_1)\subseteq X^s(z,\epsilon_X) \cap X^u(z,\epsilon_X)= \lbrace z \rbrace;$$ 
that is, $y=z$.  It follows that
\begin{equation}\label{duxx}
d^u([[x]],[[x]])=0.
\end{equation}

For $x$ near $\partial\mathcal{P}$, we want to enlarge $[[x]]$ using $\langle x, [[y]]\rangle$ for some $y\in \partial\mathcal{P}$ (see Figure \ref{enlarge}).  

\begin{figure}[h]
\psset{xunit=2cm,yunit=2cm}
\begin{center}
\begin{pspicture}(0,0)(3,2)
\pspolygon[linewidth=.5pt](0,0)(1.5,0)(1.5,1)(0,1)
\pspolygon[linewidth=.5pt](.7,1)(2.7,1)(2.7,1.8)(.7,1.8)
\pspolygon[linewidth=.5pt](1.8,0)(3,0)(3,1)(1.8,1)
\psline[linewidth=1.5pt](0,1)(3,1)
\rput(2.5,.85){$[[y]]$}
\psdot(2,1)
\rput(2, 0.85){$y$}
\psline[linewidth=1.5pt](.7, 1.1)(2.7, 1.1)
\psdot (1, 1.1)
\rput(1, 1.2){$x$}
\rput(1.5, 1.25){$[[x]]$}
\psline[linestyle=dotted, dotsep=1pt](0,1.1)(3, 1.1)
\rput(3.35, 1.2){$\langle x, [[y]]\rangle $}
\end{pspicture}
\end{center}
\caption{Enlarging the $[[\cdot]]$'s near $\partial\mathcal{P}$}
\label{enlarge}
\end{figure}
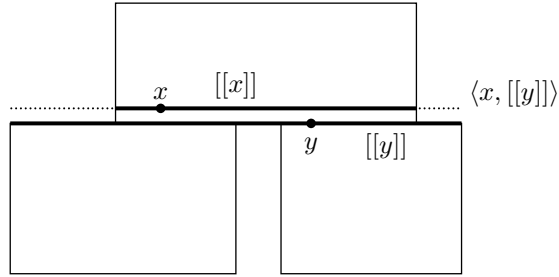

Let $B^o(x,\epsilon)$ denote the open ball around $x$ of radius $\epsilon$.

The collection $\lbrace B^o([[x]],\beta(\frac{1}{4}\epsilon_{[[x]]})) \; | \; x\in \partial \mathcal{P} \rbrace$ clearly covers $\partial \mathcal{P} $, so there exists a finite subcover with centers $[[x_1]], \cdots, [[x_L]]$.  That is, for any $x\in \partial \mathcal{P} $ there exists $1\leq l \leq L$ such that $d(x, [[x_l]])< \beta(\frac{1}{4}\epsilon_{[[x_l]]})$.  Hence, by Lemma \ref{epsilonx}, $$d^u([[x]], [[x_l]])\leq d^u(\langle x, [[x_l]] \rangle , [[x_l]]) \leq \tfrac{1}{4}\epsilon_{[[x_l]]}.$$
Denote $\mathcal{C}\equiv \bigcup_{l=1}^L B^o([[x_l]],\beta(\frac{1}{4}\epsilon_{[[x_l]]})$.

Define 
$$\lceil x \rceil= [[x]] \bigcup \left(\bigcup \lbrace \langle x, [[x_l]] \rangle \; \vert  \; 1\leq l \leq L, \; d^u([[x]], [[x_l]])<\frac{1}{2}\epsilon_{[[x_l]]}\rbrace \right) .$$

Denote $\mathcal{O}(\partial\mathcal{P} )=\lbrace x \in X \; | \; d^u([[x]], [[x_l]])< \frac{1}{2} \epsilon_{[[x_l]]}\; \mathrm{for \; some} \; 1\leq l \leq L\rbrace.$ 
We observed above that $\partial\mathcal{P}\subseteq \mathcal{C} \subseteq \mathcal{O}(\partial\mathcal{P} )$, and so $\mathcal{O}(\partial\mathcal{P} )^c \subseteq \mathcal{C}^c\subseteq \mathrm{Int}(\mathcal{P} )$.  In addition, notice that
$$\lceil x \rceil =[[x]]$$ 
if $x\in \mathcal{O}(\partial\mathcal{P})^c$.  

In fact, it is easy to check that each $x\in \mathcal{O}(\partial\mathcal{P})$ is only enlarged by at most one $[[x_m]]$.  

For $x,y\in X$, let $P(x,y)$ consist of finite paths $p=(p_0, p_1, \cdots, p_I)$ satisfying $p_0=x$, $p_I=y$, and $X^u([[p_i]],[[p_{i+1}]], \epsilon_1)\neq \varnothing$ for each $0\leq i < I$.
We define the length of a path $p=(p_0, p_1, \cdots, p_I)$ to be
$$l(p)=\sum_{i=0}^{I-1}d^u(\lceil p_i \rceil,\lceil p_{i+1} \rceil).$$

\begin{prop}
$\delta([[x]],[[y]])=\mathrm{inf} \lbrace 1, l(p) \; | \; p\in P(x,y) \rbrace$ defines a metric on $X/_\sim$.  Moreover, $\delta$ gives the quotient topology on $X/_\sim$, so that $(X/_\sim, \delta)$ is compact.  
\end{prop}

\begin{proof}[Outline of Proof.]
Due to the length and highly technical nature of this argument, we simply give an outline.  %For the full details, we refer the reader to \cite{thesis}.
 
\noindent i) $[[x]]=[[y]]$  $\Rightarrow$ $\delta([[x]],[[y]])=0$:

If $[[x]]=[[y]]$ then $\lceil x \rceil =\lceil y \rceil$ and $(x,y)\in P(x,y)$.  So $\delta([[x]],[[y]])\leq d^u( \lceil x \rceil ,\lceil y \rceil)$.  Let's show that $d^u( \lceil x \rceil ,\lceil x \rceil)=0$ for all $x\in X$.  We remarked above that either $\lceil x \rceil = [[x]]$ or $\lceil x \rceil= \langle x,[[x_m]] \rangle$ for some $1\leq m \leq L$ with $d^u([[x]],[[x_m]])<\frac{1}{2}\epsilon_{[[x_m]]}$.   For the first case, we've already shown that $d^u( [[x]] ,[[x]])=0$ for all $x$.  For the latter case, it can be shown that 
\[X^u(\langle x,[[x_m]] \rangle, \langle x,[[x_m]] \rangle,\epsilon_1)=\lbrace (f^{-K}[f^K(x),f^K(x_m')],f^{-K}[f^K(x),f^K(x_m')]) \; | \]
\hfill $x_m'\in [[x_m]] \rbrace,$

\noindent so that $d^u(\langle x,[[x_m]] \rangle, \langle x,[[x_m]] \rangle)=0$.  

\medskip
\noindent ii)  $\delta([[x]],[[y]])=0$ $\Rightarrow$ $[[x]]=[[y]]$:

This is the only difficult element of the proof, and is broken into two separate cases.  

\textit{Case 1:}  $x\notin \mathcal{O}(\partial\mathcal{P})$

\noindent In this case it is clear that $\delta([[x]],[[y]])=0$ implies that $x$ and $y$ are in the interior of the same rectangle.  Moreover, for $p\in P(x,y)$ with $l(p)$ small enough, it can be shown that $d(x,[y,x])\leq l(p).$  As a result, we have $x=[y,x]\in X^s(y,\epsilon_X)$ and hence $x \thickapprox y$.

\textit{Case 2:} $x\in \mathcal{O}(\partial\mathcal{P})$

\noindent  In this case we make use of the fact that $\lceil x \rceil=\langle x,[[x_m]]\rangle$ for some $1\leq m \leq \nolinebreak L$ with $d^u([[x]],[[x_m]]) <\frac{1}{2}\epsilon_{[[x_m]]}$.  It can be shown that for $p \in P(x,y)$ with \linebreak $l(p)< \frac{1}{2}\epsilon_{[[x_m]]} - d^u([[x]],[[x_m]])$, we have $d(y,[[x]])\leq l(p)$; it follows that $y\in [[x]]$.  

\medskip
\noindent iii) $\delta([[x]],[[y]])=\delta([[y]],[[x]])$:

This follows immediately from the observation that $(p_0,\cdots, p_I)\in P(x,y)$ if and only if $(p_I,\cdots, p_0)\in P(y,x)$, and the symmetry of $d^u$.  
 
 \medskip 
 \noindent iv) $\delta([[x]],[[y]])\leq \delta([[x]],[[z]])+\delta([[z]],[[y]])$:

If $\delta([[x]],[[z]])=1$ or $\delta([[z]],[[y]])=1$, this holds trivially.  So assume that $\delta([[x]],[[z]])<1$ and $\delta([[z]],[[y]])<1$.  If $p=(p_0, \cdots, p_I)\in P(x,z)$ and \linebreak $q=(q_0, \cdots, q_J)\in P(z,y)$, then $(p_0, \cdots , p_I=q_0, q_1, \cdots q_J)\in P(x,y)$, that is,  $P(x,y)\neq \varnothing$.  Therefore
\begin{align*}
\delta([[x]],[[y]]) & \leq  \mathrm{inf}\lbrace l(p) \; | \; p\in P(x,y) \rbrace \\
& \leq  \mathrm{inf}\lbrace l(p) \; | \; p=(p_0,\cdots, p_I)\in P(x,y), \; p_i=z \; \mathrm{for \; some} \; 0\leq i \leq I \rbrace \\
& =  \mathrm{inf}\lbrace l(p_0, \cdots, p_i) + l(p_i, \cdots, p_I)   \; | \; p=(p_0,\cdots, p_I)\in P(x,y), \; p_i=z  \rbrace \\
& =  \mathrm{inf}\lbrace l(p') + l(p'')   \; | \; p'\in P(x,z), \; p''\in P(z,y) \rbrace \\
& =  \mathrm{inf}\lbrace l(p') \; | \; p'\in P(x,z) \rbrace +  \mathrm{inf}\lbrace l(p'')   \; | \;  p''\in P(z,y) \rbrace \\
& =  \delta([[x]],[[z]])+\delta([[z]],[[y]]).
\end{align*}

 \medskip 
 \noindent v) $\delta$ gives the quotient topology on $X/_\sim$
 
Let $\mathcal{T}_q$ denote the quotient topology on $X/_\sim$.  The quotient map $(X,d) \rightarrow \nolinebreak (X/_\sim, \delta)$ can be shown to be  continuous by considering the two separate cases of part ii) above.  Since $\mathcal{T}_q$ is defined to be the finest topology on $X$ which makes the quotient map $X \rightarrow X/_\sim$ continuous, it follows that the identity map $\mathrm{id}:(X/_\sim, \mathcal{T}_q) \rightarrow \nolinebreak (X/_\sim, \delta)$ is continuous.  Since this identity map is a bijection from a compact space to a Hausdorff space, it follows that it is in fact a homeomorphism.  That is, the two topologies are the same.  
 
\end{proof}

The final element in the construction of our quotient space is the definition of an appropriate mapping.  
Let us show that the natural mapping $\alpha : X/_\sim \rightarrow X/_\sim$ given by 
$$\alpha([[x]])=[[f(x)]]$$
is well-defined.

We begin by showing that $y\in [[x]]$ implies $f(y)\in [[f(x)]]$.  First, consider the case where $x\in \mathrm{Int}(R_i)\cap f^{-1}(\mathrm{Int}(R_j))$, and suppose that $y \thickapprox x$.  Then $y\in \nolinebreak X^s(x, R_i)$, so by the definition of a Markov partition, it follows that 
$$f(y)\in f(X^s(x, R_i))\subseteq X^s(f(x), R_j);$$
that is, $f(y)\thickapprox f(x)$.  Since $\sim$ is generated by $\thickapprox$, we also have $x\sim y$ implies that \linebreak $f(x)\sim f(y)$.

Now choose any $x\in X$, and suppose $y \thickapprox x$.  Then $x,y\in R_i$ for some $R_i\in \mathcal{P}$.  
Bowen \cite{Bowen} proves that $f(x)\in R_j$ for some $j$ with $\mathrm{Int}(R_i)\cap f^{-1}(\mathrm{Int}(R_j))\neq \varnothing$, and moreover that
$$f(X^s(x, R_i))\subseteq X^s(f(x), R_j).$$
Therefore $f(y)\thickapprox f(x)$.  Since $\sim$ is generated by $\thickapprox$, we also have $x\sim y$ implies $f(x)\sim f(y)$.  

That $\alpha$ is surjective follows immediately from the surjectivity of $f$.  

\begin{prop}
$\alpha: X/_\sim \rightarrow X/_\sim$ is continuous.
\end{prop}

\begin{proof}[Outline of Proof]
Let $[[y_n]]\rightarrow [[y]]$.  We want to show that $[[f(y_n)]]\rightarrow [[f(y)]]$.

First, it can be shown that there exist
\begin{equation}\label{pair}
(a_n,b_n)\in X^u([[f(y_n)]],[[f(y)]],\epsilon_1) \; \mathrm{such \; that} \; d(a_n,b_n)\rightarrow 0 \; \mathrm{as} \; n \rightarrow \infty
\end{equation}
by separately considering the two cases $y\notin \mathcal{O}(\partial\mathcal{P})$ and $y\in \mathcal{O}(\partial\mathcal{P})$.

Then it can be shown that \eqref{pair} implies  $[[f(y_n)]]\rightarrow [[f(y)]]$ by considering the two cases 
$f(y)\notin \mathcal{O}(\partial\mathcal{P})$ and $f(y)\in \mathcal{O}(\partial\mathcal{P})$.  In the first case we get 
\begin{align*}
\delta([[f(y_n)]],[[f(y)]]) & \leq  d^u(\lceil f(y_n)\rceil ,\lceil f(y) \rceil)\\
& = \mathrm{sup} \lbrace d([f(y_n),z], z) \; | \; z \in [[f(y)]]\rbrace \\
& = \mathrm{sup} \lbrace d([a_n,z],[b_n,z]) \; | \; z \in [[f(y)]]\rbrace \\
& \rightarrow  0.
\end{align*}
And in the second case we get 
\begin{align*}
\delta([[f(y_n)]],[[f(y)]]) & \leq  d^u(\lceil f(y_n)\rceil ,\lceil f(y)\rceil) \\
& \leq  \mathrm{sup}\lbrace d(f^{-K}[f^K(f(y_n),f^K(x_m'')],f^{-K}[f^K(f(y)),f^K(x_m'')]) \; | \\
&\qquad \qquad \hfill x_m'' \in [[x_m]]\rbrace \\
& =  \mathrm{sup}\lbrace d(f^{-K}[f^K(a_n),f^K(x_m'')],f^{-K}[f^K(b_n),f^K(x_m'')]) \; | \\
& \qquad \qquad x_m'' \in [[x_m]]\rbrace \\
& \rightarrow  0.
\end{align*}

%For the complete details, see \cite{thesis}.

\end{proof}

\subsection{The Quotient Space Satisfies Axioms 1 and 2}

We have already shown that $(X/_\sim, \delta)$ is a compact metric space, and that the mapping $\alpha:X/_\sim \rightarrow X/_\sim$ is continuous and surjective.  

Choose $K$ as in Lemma \ref{K} and let $\gamma=\lambda$, the expansive constant for the Smale space $(X,d,f)$.  We will show that there exists $\beta>0$ such that 
\begin{description}
\item[Axiom 1]{\textit{if $\delta([[x]],[[y]])\leq\beta$ then $$\delta([[f^K(x)]], [[f^K(y)]]) \leq \gamma^K \delta([[f^{2K}(x)]], [[f^{2K}(y)]]),$$ and}}
\item[Axiom 2]{\textit{for all $[[x]]\in X/_\sim$ and $0<\epsilon\leq\beta$, 
$$\alpha^{K}(B([[f^K(y)]], \epsilon)) \subseteq \alpha^{2K}(B([[y]], \gamma\epsilon).$$}}
\end{description}

\begin{lemma}\label{subset}
For any $[[x]]\in X/_\sim$, $f^K\lceil x \rceil \subseteq [[f^K(x)]]$.
\end{lemma}

\begin{proof} 
\textit{Case 1:}  $x\notin \mathcal{O}(\partial\mathcal{P})$

Then $\lceil x \rceil =[[x]]$, so that $f^K\lceil x \rceil = f^K [[x]]\subseteq [[f^K(x)]]$ by Lemmas \ref{epsilon_0} and \ref{K}.

\noindent \textit{Case 2:} $x\in \mathcal{O}(\partial\mathcal{P})$

Then $\lceil x \rceil = \langle x, [[x_m]] \rangle$ for some $1\leq m \leq L$ with $d^u([[x]],[[x_m]])<\frac{1}{2}\epsilon_{[[x_m]]}$.  Choose  $x_m'\in [[x_m]]$.  
Recall from \eqref{smallenough} that 
$$d(f^K(x),f^K(x_m'))\leq 3\epsilon_0'\leq \epsilon_0''.$$
By our choice of $\epsilon_0''$, it follows that 
$$[f^K(x),f^K(x_m')]\in X^s(f^K(x),\epsilon_0).$$
So by Lemma \ref{epsilon_0}, we have $[f^K(x),f^K(x_m')]\thickapprox f^K(x).$
\end{proof}

Lemma \ref{subset} can be used to prove that there exists $\eta_1>0$ such that 
\begin{equation}\label{6.26}
\delta([[x]],[[y]])<\eta_1\; \mathrm{implies}\; X^u([[f^K(x)]],[[f^K(y)]],\beta(\epsilon_1))\neq \varnothing. 
\end{equation}
 %For a detailed proof, see Lemma 6.26 of \cite{thesis}.
By the uniform continuity of $\alpha$ and $f^{-1}$, there exists 
$$0<\eta_2<\eta_1$$ 
such that $\delta([[x]],[[y]])\leq \eta_2$ implies $\delta(\alpha^K[[x]],\alpha^K[[y]])<\eta_1$, and such that \linebreak $d(x,y)\leq \eta_2$ implies $d(f^{-K}(x),f^{-K}(y))\leq \epsilon_1$.  It follows from \eqref{6.26} that 
\begin{equation}\label{6.27}
\delta([[x]],[[y]])\leq \eta_2 \; \mathrm{implies} \; d^u([[f^K(x)]],[[f^K(y)]])\leq \delta([[f^K(x)]],[[f^K(y)]]). 
\end{equation}
%For a detailed proof, see Corollary 6.27 of \cite{thesis}.

Choose
$$0<\eta_3<\eta_2$$ 
such that $\delta([[x]],[[y]])\leq \eta_3$ implies $\delta(\alpha^K[[x]],\alpha^K[[y]])\leq \eta_2$, and such that $d(x,y)\leq \nolinebreak\eta_3$ implies that $d([x,z],[y,z])\leq \eta_2$ for all $z$ such that $(x,z),(y,z)\in \mathrm{domain}([\cdot, \cdot ])$.   We will show that $\beta=\eta_3$ satisfies Axioms 1 and 2.  

For Axiom 1, suppose that $\delta([[x]],[[y]])\leq \eta_3<\eta_1$.  By \eqref{6.26}, we have that
$(f^K(x),f^K(y))\in P(f^K(x),f^K(y))$,
so that 
\begin{equation}\label{no1}
\delta([[f^K(x)]], [[f^K(y)]])\leq d^u(\lceil f^K(x) \rceil , \lceil f^K(y) \rceil ) .
\end{equation}
Since $\delta([[x]],[[y]])\leq \eta_3$ implies $\delta([[f^{K}(x)]],[[f^{K}(y)]])\leq \eta_2$, we have by \eqref{6.27} that
\begin{equation}\label{no3}
d^u([[f^{2K}(x)]],[[f^{2K}(y)]])\leq \delta([[f^{2K}(x)]],[[f^{2K}(y)]]).
\end{equation}
And by Lemma \ref{subset}, 
\begin{equation}\label{no2}
\lceil f^K(x) \rceil\subseteq f^{-K}[[f^{2K}(x)]] \; \mathrm{and} \; \lceil f^K(y) \rceil\subseteq  f^{-K}[[f^{2K}(y)]].
\end{equation}
Combining \eqref{no1}, \eqref{no2}, and \eqref{no3}, we get
\begin{align*}
\delta([[f^K(x)]], [[f^K(y)]])&\leq d^u(\lceil f^K(x) \rceil , \lceil f^K(y) \rceil ) \\
&\leq d^u(f^{-K}[[f^{2K}(x)]], f^{-K}[[f^{2K}(y)]])\\
& = \mathrm{sup}\lbrace d(u,v) \; | \; (u,v)\in X^u(f^{-K}[[f^{2K}(x)]], f^{-K}[[f^{2K}(y)]],\epsilon_1) \rbrace \\
& = \mathrm{sup}\lbrace d(f^{-K}(u),f^{-K}(v)) \; | \\
&\qquad \qquad \qquad \qquad \qquad (u,v)\in X^u([[f^{2K}(x)]], [[f^{2K}(y)]],\epsilon_1) \rbrace \\
& = \mathrm{sup}\lbrace \lambda^K d(u,v) \; | \; (u,v)\in X^u([[f^{2K}(x)]], [[f^{2K}(y)]],\epsilon_1) \rbrace \\
& = \lambda^K d^u([[f^{2K}(x)]], [[f^{2K}(y)]])\\
& \leq \lambda^K \delta([[f^{2K}(x)]], [[f^{2K}(y)]]).\\
\end{align*}

\medskip

Axiom 2 will take a little more work.  We want to prove that 
$$\alpha^{K}(B([[f^K(y)]], \epsilon)) \subseteq \alpha^{2K}(B([[y]], \gamma\epsilon)$$
for all $0<\epsilon\leq\eta_3$.

Let $\delta([[z]], [[f^K(y)]])\leq \eta_3$.  Recall that we denoted our finite cover of $\partial(\mathcal{P})$ by $\mathcal{C}\equiv \bigcup_{l=1}^L B^o([[x_l]],\beta(\frac{1}{4}\epsilon_{[[x_l]]})$.  We will consider two separate cases.  

\medskip

\noindent \textit{Case 1:}  $f^K(y) \notin \mathcal{C}$

It is not hard to show that $f^K(y)$ and $z$ are in the interior of the same rectangle and that 
$$(f^K(y),[z,f^K(y)])\in X^u([[f^K(y)]],[[z]],\eta_3)\subseteq X^u([[f^K(y)]],[[z]],\eta_2).$$
So by our choice of $\eta_2$,
$$(y,f^{-K}([z,f^K(y)]))\in X^u([[y]],[[f^{-K}([z,f^K(y)])]],\epsilon_1),$$
hence $\delta([[y]],[[f^{-K}([z,f^K(y)])]])  \leq d^u(\lceil y\rceil, \lceil f^{-K}([z,f^K(y)])\rceil).$
Moreover, we have 
$$d([[f^K(y)]],[[z]])\leq d(f^K(y),[z,f^K(y)])\leq \eta_3 \leq \epsilon_1'.$$  
We leave the proof that
\begin{align*}
d^u([[f^K(y)]],[[z]]) & =\mathrm{sup}\lbrace d(u, [z,u])\; | \; u\in [[f^K(y)]] \rbrace \\
& =\mathrm{sup}\lbrace d([f^K(y),u], [\,[z,f^K(y)],u])\; | \; u\in [[f^K(y)]] \rbrace\\
& \leq \eta_2
\end{align*}
as an exercise for the interested reader.  
 
So by Lemma \ref{subset},
\begin{align*}
\delta([[y]],[[f^{-K}([z,f^K(y)])]]) & \leq d^u(\lceil y\rceil, \lceil f^{-K}([z,f^K(y)])\rceil) \\
& \leq d^u(f^{-K}[[f^K(y)]], f^{-K}[[\,[z,f^K(y)]\,]]) \\
& = d^u(f^{-K}[[f^K(y)]], f^{-K}[[z]]) \\
& \leq \lambda^K d^u([[f^K(y)]],[[z]]) \\
& \leq \eta_2
\end{align*}

It follows by \eqref{6.27} that $d^u([[f^K(y)]],[[z]]) \leq \delta([[f^K(y)]],[[z]]).$  Therefore
\begin{align*}
\delta([[y]],[[f^{-K}([z,f^K(y)])]]) & \leq \lambda^K d^u([[f^K(y)]],[[z]]) \\
& \leq \lambda^K \delta([[f^K(y)]],[[z]]).
\end{align*}
Moreover, we have $\alpha^{2K}([[f^{-K}([z,f^K(y)])]])=\alpha^K([[\,[z,f^K(y)]\,]])=\alpha^K[[z]]$.

\medskip

\noindent \textit{Case 2:} $f^K(y)\in \mathcal{C}$

Then $d(f^K(y),[[x_m]])<\beta(\frac{1}{4}\epsilon_{[[x_m]]})$ for some $1\leq m \leq L$.  %So by Corollary \ref{deltax}, we have
%$d^u([[f^K(y)]],[[x_m]])\leq \frac{1}{4}\epsilon_{[[x_m]]}$. 

Let $p=(p_0, \cdots, p_I) \in P(f^K(y),z)$ such that $l(p)< \eta_1$.  
Since $\mathcal{C}\subseteq \mathcal{O}(\partial\mathcal{P})$, %and $l(p)+ d^u([[f^K(y)]],[[x_m]]) <\frac{1}{4}\epsilon_2+ \frac{1}{4}\epsilon_{[[x_m]]} =\frac{1}{2}\epsilon_{[[x_m]]}$, we can apply Lemma \ref{smallp2} to get $d^u([[p_i]],[[x_m]]) \leq d^u([[f^K(y)]],[[x_m]]) +l(p)<\frac{1}{2}\epsilon_{[[x_m]]}$ for all $i=0, \cdots, I$.  Therefore 
it follows with a little work that $[[p_i]]\subseteq \langle p_i, [[x_m]] \rangle \subseteq \lceil p_i \rceil $ for all $i=0, \cdots, I$.  

So we have 
$$f^K(y)=f^{-K}[f^{2K}(y), f^K(x_m')]$$ for some $x_m'\in [[x_m]]$.  Define 
$$u_i=f^{-K}[f^{K}(p_i), f^K(x_m')]$$
for all $i=0, \cdots, I$.  
Then by Lemma \ref{subset},
$$f^K(u_i)\in f^K\langle p_i, [[x_m]] \rangle \subseteq f^K\lceil p_i \rceil \subseteq [[f^K(p_i)]]$$
for all $i=0, \cdots, I$.  

It is easy to check that $f^{-K}[f^{K}(u_I), f^K(x_m'')])=f^{-K}[f^{K}(p_I), f^K(x_m'')])$ for all $x_m''\in [[x_m]]$, and that 
\begin{equation}\label{extra}
(f^K(y), u_I)\in X^u([[f^K(y)]],[[u_I]],\eta_1).
\end{equation}
Therefore
\begin{align*}
d^u([[f^K(y)]],[[u_I]]) & \leq d^u( \langle f^K(y), [[x_m]] \rangle , \lceil u_I \rceil) \\
&=\mathrm{sup}\lbrace d(f^{-K}[f^{2K}(y), f^K(x_m'')],f^{-K}[f^{K}(u_I), f^K(x_m'')]) \; | \\
& \qquad \qquad  x_m''\in [[x_m]]\rbrace\\
&=\mathrm{sup}\lbrace d(f^{-K}[f^{2K}(y), f^K(x_m'')],f^{-K}[f^{K}(z), f^K(x_m'')]) \; | \\
& \qquad \qquad x_m''\in [[x_m]]\rbrace\\
&\leq \mathrm{sup}\lbrace \sum d(f^{-K}[f^{2K}(p_i), f^K(x_m'')],f^{-K}[f^{K}(p_{i+1}), f^K(x_m'')])  \; | \\
& \qquad \qquad x_m''\in [[x_m]]\rbrace\\
& \leq \sum_i d^u(\langle p_i, [[x_m]], \langle p_{i+1}, [[x_m]] \rangle ) \\
& \leq \sum_i d^u(\lceil p_i \rceil, \lceil p_{i+1}\rceil) \\
& = l(p),
\end{align*}
and so $d^u([[f^K(y)]], [[u_I]]) \leq \mathrm{inf}\lbrace l(q) \; | \; q\in P(f^K(y),z),\; l(q)< \eta_1 \rbrace =\delta([[f^K(y)]], [[z]]).$ 

Moreover, \eqref{extra} gives
$$(y, f^{-K}(u_I))\in X^u([[y]],[[f^{-K}(u_I)]],\epsilon_1)$$
by our choice of $\eta_1$.  

So by Lemma \ref{subset}, 
\begin{align*}
\delta([[y]], [[f^{-K}(u_I)]]) & \leq d^u( \lceil y \rceil, \lceil f^{-K}(u_I)\rceil ) \\
& \leq d^u(f^{-K}[[f^K(y)]], f^{-K}[[u_I]]) \\
& \leq \lambda^K d^u( [[f^K(y)]], [[u_I]])\\
& \leq \lambda^K \delta([[f^K(y)]], [[z]]).
\end{align*}

And we had $f^{K}(u_I)\in f^{K}(\lceil z \rceil) \subseteq [[f^{K}(z)]]$, so that
$$\alpha^{2K}[[f^{-K}(u_I)]]=[[f^K(u_I)]]=[[f^K(z)]]=\alpha^K[[z]].$$

\subsection{Topological Conjugacy}

Since we have shown that $(X/_\sim, \alpha, \delta)$ satisfies Axioms 1 and 2, we use the notation of Section \ref{thmApf} for the inverse limit associated with this system.  That is, we denote
\[\widehat{X/_\sim} =\underleftarrow{\mathrm{lim}}\;X/_\sim \stackrel{\alpha}{\longleftarrow} X/_\sim \stackrel{\alpha}{\longleftarrow}  \cdots.\]
Moreover, $\hat{\alpha}:\widehat{X/_\sim} \rightarrow \widehat{X/_\sim}$ is given by 
\begin{align*}
\hat{\alpha}([[x_0]], [[x_1]], [[x_2]], \cdots ) & = (\alpha([[x_0]]), \alpha([[x_1]]), \alpha([[x_2]]), \cdots) \\
& =([[f(x_0)]], [[x_0]], [[x_1]], \cdots )  ,
\end{align*}
and
$$\hat{\delta}(\mathbf{x},\mathbf{y})=\sum_{k=0}^{K-1}\gamma^{-k}\delta'(\hat{\alpha}^{-k}(\mathbf{x}),\hat{\alpha}^{-k}(\mathbf{y})),$$
where 
$$\delta'(\textbf{x},\textbf{y})=\mathrm{sup}\lbrace \gamma^n \delta([[x_n]],[[y_n]]) \; | \; n\geq 0 \rbrace$$
and $K\geq 1$ and $0<\gamma<1$ are the Axioms 1 and 2 constants for $(X/_\sim, \alpha, \delta)$ .
 
Define $\omega : X \rightarrow \widehat{X/_\sim}$ by
$$\omega (x) = ([[x]], [[f^{-1}(x)]], [[f^{-2}(x)]], \cdots ).$$

\begin{lemma}
$\omega: (X,d) \rightarrow (\widehat{X/_\sim},\delta')$ is a homeomorphism.
\end{lemma}

\begin{proof}
Since $(X,d)$ is compact and $(\widehat{X/_\sim},\delta')$ is a metric space, it suffices to prove that $\omega$ is continuous and bijective.  

\medskip

We begin with continuity.  Let $\epsilon>0$.  Choose $N\in\mathbb{N}$ such that $\lambda^N<\epsilon$.  Since the quotient map $X \rightarrow X/_\sim$ is uniformly continuous, there exists $\epsilon'>0$ such that $\delta([[x]],[[y]])<\epsilon$ if $d(x,y)<\epsilon'$.  And there exists $\epsilon''>0$ such that $d(x,y)<\epsilon''$ implies that $d(f^{-n}(x),f^{-n}(y))<\epsilon'$ for $n=0, \cdots, N-1$.  

So let $d(x,y)<\epsilon''$.  Then for $n=0, \cdots, N-1$ we have
$$\lambda^n \delta([[f^{-n}(x)]],[[f^{-n}(y)]]) < \lambda^n \epsilon < \epsilon.$$
And for $n\geq N$, we have
$$\lambda^n \delta([[f^{-n}(x)]],[[f^{-n}(y)]])\leq \lambda^n < \epsilon.$$
Hence $\delta'(\omega(x),\omega(y))=\mathrm{sup}\lbrace \lambda^n \delta([[f^{-n}(x)]],[[f^{-n}(y)]]) \; | \; n\geq 0 \rbrace\leq \epsilon.$

\medskip
For surjectivity, let $\textbf{z}=([[z_0]], [[z_1]], \cdots ) \in \widehat{X/_\sim}$.  Observe that for all $N\in \mathbb{N}$ and each \linebreak $0\leq m \leq N$, 
$$[[z_m]]=\alpha^{N-m}([[z_N]])=[[f^{N-m}(z_N)]].$$
Hence 
\begin{align*}
\delta'(\omega(f^N(z_N)),\textbf{z}) & =  \mathrm{sup} \lbrace \lambda^n\delta([[f^{N-n}(z_N)]],[[z_n]]) \; | \; n\geq 0 \rbrace \\
& =  \mathrm{sup} \lbrace \lambda^n\delta([[f^{N-n}(z_N)]],[[z_n]]) \; | \; n> N \rbrace \\
& < \lambda^N.
\end{align*}
That is, $\omega(f^N(z_N))\rightarrow \textbf{z}$ as $N \rightarrow \infty$.

However, since $(f^N(z_N))$ is a sequence in the compact space $X$, it has a convergent subsequence $f^{N_k}(z_{N_k}) \rightarrow y$.  By the continuity of $\omega$, 
$$\omega(f^{N_k}(z_{N_k})) \rightarrow \omega(y),$$
hence $\textbf{z}=\omega(y)$.  

\medskip
And finally, to prove injectivity suppose $\omega(x)=\omega(y)$.  Then $f^{-n}(x)\sim f^{-n}(y)$ for all $n\geq 0$.  In particular, $$f^{-(K+n)}(x)\sim f^{-(K+n)}(y)$$ 
for all $n\geq 0$, so that $f^{-n}(x)\in X^s(f^{-n}(y), \epsilon_0')$ by Lemma \ref{K}.  This implies
$$x\in X^s(y, \lambda^n\epsilon_0'),$$
and hence $d(x,y)\leq \lambda^n\epsilon_0'$, for all $n\geq 0$.  So $x=y$.  
\end{proof}

Now let's show that the following diagram commutes:
\[\begin{CD}
X @>f>> X\\
@VV\omega V @VV\omega V\\
\widehat{X/_\sim} @>\widehat{\alpha}>> \widehat{X/_\sim}
\end{CD}\]

Let $x\in X$.  Then 
\begin{align*}
\omega \circ f(x) & = \omega(f(x))\\
& = ([[f(x)]], [[x]], [[f^{-1}(x)]], \cdots )\\
& = \hat{\alpha}([[x]], [[f^{-1}(x)]], \cdots )\\
& = \hat{\alpha}\circ \omega(x).
\end{align*}

Therefore $(X,d,f)$ and $(\widehat{X/_\sim}, \hat{\delta}, \hat{\alpha})$ are topologically conjugate.

\bibliographystyle{amsplain}

\end{document}